\documentclass[11pt]{article}
\usepackage{amssymb, amsmath, amsthm, oldgerm,srcltx}
\usepackage{graphicx}
\usepackage{xcolor}
\newtheorem{thm}{Theorem}[section]
\newtheorem{lem}{Lemma}[section]
\newtheorem{cor}[lem]{Corollary}

\numberwithin{equation}{section}

\newcommand{\eps}{\varepsilon}

\newcommand{\E}{\mathbf{E}\,}

\newcommand{\R}{\mathbf{R}}

\newcommand{\im}{\mathrm{Im}\;\!}
\newcommand{\Tr}{\mathrm{Tr}\;\!}

\newenvironment{Proof of}{\removelast
 skip\par\medskip
\noindent{\em Proof of} \rm}{\penalty-20\null\hfill$\square$\par\medbreak}

\def\be{\begin{equation}}
\def\en{\end{equation}}
\def\bee{\begin{eqnarray*}}
\def\ene{\end{eqnarray*}}

\def\E{{\bf E}}

\def\R{{\mathbb R}}

\def\Tr{{\rm Tr}\,}
\def\im{{\rm Im}\,}

\def\<{\left<}
\def\>{\right>}

\def\1{{\bf 1}}
\def\4{\kern1pt}

\begin{document}
\bibliographystyle{}

\vspace{1in}
 \date{}
\title
{Optimal bounds for convergence of  expected spectral distributions to the semi-circular law\\   
for the $4+\eps$ moment ensemble}
\author{{\bf F. G\"otze}\\{\small Faculty of Mathematics}
\\{\small University of Bielefeld}\\{\small Germany}
\and {\bf A. Tikhomirov}\\{\small Department of Mathematics
}\\{\small Komi Science Center of Ural Division of RAS,}\\{\small Syktyvkar State University}
\\{\small  Russia}
}
\maketitle
 \footnote{Research supported   by SFB 701 ``Spectral Structures and Topological Methods in Mathematics'' University of Bielefeld.
  Tikhomirov's research supported   by grants  RFBR N~14-01-00500  and   by Program of Fundamental Research UD of RAS, Project ¹ 15-16-1-3}

\date{}

\maketitle
\begin{abstract}
This paper extends a previous bound of order $O(n^{-1})$ of the authors \cite{GT:2015}
for the rate of convergence in Kolmogorov distance
of the expected  spectral distribution of a Wigner random matrix ensemble to the semicircular law. We relax the moment
conditions for  entries of the Wigner matrices  from order $8$ to order  $4+ \eps$ for an arbitrary small $\eps>0$.
\end{abstract}
\maketitle

\setcounter{equation}{0}

\section{Introduction}
Consider a family $\mathbf X = \{X_{jk}\}$, $1 \leq j \leq k \leq n$,
of independent real random variables defined on some probability space
$(\Omega,{\textfrak M},\Pr)$, for any $n\ge 1$. Assume that $X_{jk} = X_{kj}$, for
$1 \leq k < j \leq n$, and introduce the symmetric matrices
\begin{displaymath}
 \mathbf W = \ \frac{1}{\sqrt{n}} \left(\begin{array}{cccc}
 X_{11} &  X_{12} & \cdots &  X_{1n} \\
 X_{21} & X_{22} & \cdots &  X_{2n} \\
\vdots & \vdots & \ddots & \vdots \\
 X_{n1} &  X_{n2} & \cdots &  X_{nn} \\
\end{array}
\right).
\end{displaymath}

The matrix $\mathbf W$ has a random spectrum $\{\lambda_1,\dots,\lambda_n\}$ and an
associated spectral distribution function
$\mathcal F_{n}(x) = \frac{1}{n}\ {\rm card}\,\{j \leq n: \lambda_j \leq
x\}, \quad x \in \R$.
Averaging over the random values $X_{ij}(\omega)$, define the expected
(non-random) empirical distribution functions
$ F_{n}(x) = \E\,\mathcal F_{n}(x)$.
Let $G(x)$ denote the semi-circular distribution function with density
$g(x)=G'(x)=\frac1{2\pi}\sqrt{4-x^2}\mathbb I_{[-2,2]}(x)$, where $\mathbb I_{[a,b]}(x)$
denotes the indicator--function of the interval $[a,b]$. Let $\Delta_n:=\sup_x| F_n(x)-G(x)|$.
In a recent paper \cite{GT:2015} we proved the following result
\begin{thm} Let $\E X_{jk}=0$, $\E X_{jk}^2=1$.  Assume that for some $0<\varkappa\le 4$
\begin{equation}\label{moment}
 \sup_{n\ge1}\sup_{1\le j,k\le n}\E|X_{jk}|^4=: \mu_4<\infty.
\end{equation}
Assume as well that there exists a constant $D_0$ such that for all $n\ge 1$
\begin{equation}
 \sup_{1\le j,k\le n}|X_{jk}|\le D_0n^{\frac14}.
\end{equation}

Then,   there exists a positive  constant $C=C(D_0,\mu_4)$  depending on
 $D_0$ and $\mu_4$ only
such that
\begin{equation} \label{kolmog0}
\Delta_n=\sup_x|F_n(x)-G(x)|\le Cn^{-1}.
\end{equation}

\end{thm}
\begin{cor}\label{cormain}Let $\E X_{jk}=0$, $\E X_{jk}^2=1$.  Assume that
\begin{equation}\label{moment1}
 \sup_{n\ge 1}\sup_{1\le j,k\le n}\E|X_{jk}|^8=:\mu_8<\infty.
\end{equation}
 Then,   there  exists a positive  constant $C=C(\mu_8)$  depending on
 $\mu_8$  only
such that
\begin{equation} \label{kolmog1}
\Delta_n\le Cn^{-1}.
\end{equation}
\end{cor}
Here we describe some refinements of the proof of Theorem \ref{trun} in \cite{GT:2015} showing that
\begin{thm}\label{main}
Assume that for some $0<\varkappa\le4$ there exists positive constant $0<\mu_{4+\varkappa}<\infty$ such that
\begin{equation}
\sup_{j,k\le1}\E|X_{jk}|^{4+\varkappa}\le \mu_{4+\varkappa}.
\end{equation}
Then there exists a positive constant $C$ depending on $\varkappa$ and $\mu_{4+\varkappa}$ only such that
\begin{equation}
\Delta_n\le Cn^{-1}.
\end{equation}
\end{thm}
Using standard techniques similar to \cite{GT:2015},  we may reduce the problem to the following
\begin{thm}\label{trun}
Assume that there exists a constant $0<\mu_{4+\varkappa}<\infty$ such that
\begin{equation}
\sup_{j,k\ge1}\E|X_{jk}|^{4+\varkappa}\le \mu_{4+\varkappa}.
\end{equation}
Assume that for  $\alpha:=\frac2{4+\varkappa}$
\begin{equation}
|X_{jk}|\le Dn^{\alpha}.
\end{equation}
Then there exists a positive  constant $C=C(D,\mu_{4+\varkappa},\varkappa)$  depending on
 $D$, $\mu_{4+\varkappa}$  and $\varkappa$ only
such that
\begin{equation} \label{kolmog}
\Delta_n=\sup_x|F_n(x)-G(x)|\le Cn^{-1}.
\end{equation}

\end{thm} 
For a discussion of this and previous results the reader should consult the introduction of \cite{GT:2015}.\\
For the proof we nned to revise parts  of the proof of Theorem 1.3 in \cite{GT:2015}.
	The conditions of Theorem \ref{main} suffice for the remaining parts and
	only the result \cite{GT:2015}[Theorem  \ref{trun}] needs to be strengthened.
	Let us reformulate this theorem here. Given constants  $a>0, A_0>0$, (to be chosen later), let  $\frac12>\varepsilon>0$ be a sequence of positive numbers (depending on $n$)  such that
\begin{equation}\label{avcond*}
 \varepsilon^{\frac32} =2v_0a, \quad v_0:=A_0n^{-1}.
\end{equation}
Define the region
\begin{equation}
\mathbb G=\{z=u+iv:\, |u|\le 2-\varepsilon, \, v\sqrt{\gamma}\ge v_0\},\quad\gamma:=\gamma(z)=|2-|u||, \quad\varepsilon=c_0v_0^{\frac23} 
\end{equation}
(see as well \cite[Definition (1.11)]{GT:2015}).
\begin{thm}\label{stieltjesmain}
Assuming the conditions of Theorem \ref{trun}, there exist positive constants $A_0>0$ and $C=C(D,A_0,\mu_{4+\varkappa})$ depending on $D$, $A_0$ and $\mu_{4+\varkappa}$ only, such that, for 
$z\in\mathbb G$
\begin{align}\notag
 |\E m_n(z)-s(z)|\le \frac C{n v^{\frac34}}+\frac C{n^{\frac32}v^{\frac32}|z^2-4|^{\frac14}}.
\end{align}

\end{thm} 
The key ingredient  for the proof of Theorem \ref{trun} is an essential improvement of the dependence of bounds of the quantity $\E|\varepsilon_{j2}|^p$ in \eqref{defeps} 
on $v$. This is due to a considerably  improved estimate for off-diagonal entries of the resolvent matrix, see Lemmas \ref{e2} and \ref{e5}, compared to the previous bounds in \cite[Lemma 5.8]{GT:2015}.  
\section{Estimation of  moments of diagonal resolvent entries}
In what follows we shall assume $z\in\mathbb G$.
At first let us modify the proofs of Lemma 5.13 and Corollary 5.14 of \cite{GT:2015}.
We prove the following 
\begin{thm}\label{rjj}
Assuming the conditions of Theorem \ref{trun}, there exist constants $C_0$, $A_0$ and $A_1$ depending on $\varkappa$ and $D$ such that
\begin{equation}
\E|R_{jj}|^p\le C_0^p,
\end{equation}
for $p\le A_1(nv)^{\frac{1-2\alpha}2}$ and $v\ge A_0n^{-1}$.
\end{thm}
In order to prove this result we shall 
need the following Lemmas. Let $s=2^{\frac2{1-2\alpha}}$.
\begin{lem}\label{e1} 
Assuming the conditions of Theorem \ref{trun}, we have, for all $v\ge v_0$
\begin{equation}
\E|\varepsilon_{j1}|^{2p}\le C^pn^{-p(1-2\alpha)}.
\end{equation}
\end{lem}
\begin{proof}
Note that 
\begin{equation}
|\varepsilon_{j1}|\le |X_{jj}|/\sqrt n\le Dn^{-\frac{1-2\alpha}2}.
\end{equation}
This inequality concludes the proof of Lemma \ref{e1}.  
\end{proof}

\begin{lem}\label{e3}
Assuming condition \eqref{h0} for $v\ge v_1$, we have, for all $v\ge v_1/s$
\begin{equation}
\E|\varepsilon_{j3}|^{2p}\le C^pp^{2p}n^{-2p(1-2\alpha)}(sH_0)^{2p}.
\end{equation}
\end{lem}
\begin{proof}
Recall that
\begin{equation}
\varepsilon_{j3}=\frac1n\sum_{l\in\mathbb T_j}(X_{jl}^2-1)R^{(j)}_{ll}.
\end{equation}
Applying Rosenthal's inequality, we get
\begin{equation}
\E|\varepsilon_{j3}|^{2p}\le C^p\Bigg(p^p\mu_4^pn^{-p}\E\Big(\frac1n\sum_{l\in\mathbb T_j}|R^{(j)}_{ll}|^2\Big)^p+p^{2p}n^{-2p}\mu_{4p}\sum_{l\in\mathbb T_j}\E|R^{(j)}_{ll}|^{2p}\Bigg).
\end{equation}
By the assumptions of Theorem \ref{trun} , we get
\begin{equation}
\E|\varepsilon_{j3}|^{2p}\le C^p(sH_0)^{2p}(p^pn^{-p}+p^{2p}n^{-2p(1-2\alpha)}).
\end{equation}
Here we used that $4\alpha>1$ and $\alpha<\frac12$.
These relations conclude the proof of Lemma \ref{e3}. 
\end{proof}

\begin{lem}\label{e5}
Assume that for some $1\ge v_1\ge v_0$ there exists a sufficiently large constant $H_0$ such that for any $\mathbb J$ with cardinality $|\mathbb J|\le \log_s(nv)$
and $q\le A_1(nv)^{\frac{1-2\alpha}2}$ the inequality
\begin{equation}\label{h0}
\E|R^{(\mathbb J)}_{jk}|^q\le H_0^q
\end{equation}
holds for any $j,k\in\mathbb T\setminus\mathbb J$ and any $v\ge v_1$.
Then   this inequality still holds for any $v\ge v_1/s$ and $j\ne k\in\mathbb T\setminus\mathbb J$.
\end{lem}

\begin{proof}
We use the representation
\begin{equation}
R_{jk}^{(\mathbb J)}=-\Big(\frac1{\sqrt n}\sum_{l\in\mathbb T_{\mathbb J}}X_{jl}R^{(\mathbb J,j)}_{lk}\Big)R_{kk}.
\end{equation}
Applying Cauchy's and Rosenthal's inequalities, we get
\begin{align}
\E|R_{jk}^{(\mathbb J)}|^q&\le H_0^qs^q\Big(C^qn^{-\frac q2}\E^{\frac12}\big(\sum_{l\in\mathbb T_{\mathbb J}}|R^{(\mathbb J,j)}_{lk}|^2)^{q}+C^qq^qn^{-\frac q2}\mu_{2q}^{\frac12}\E^{\frac12}|R^{\mathbb J,j}_{kk}|^{2q}\notag\\&+
C^qq^qn^{-\frac q2}\mu_{2q}^{\frac12}\E^{\frac12}\big(\sum_{l\in\mathbb T_{\mathbb J,j,k}}|R^{(\mathbb J,j)}_{lk}|^{2q}\big)\Big).
\end{align}
Applying now \cite[Lemma 7.6]{GT:2015},  inequality \eqref{mom4} and the assumptions of Lemma \ref{e3}, we get
\begin{align}\label{11}
\E|R_{jk}^{(\mathbb J)}|^q&\le H_0^qs^q\Big(C^qq^{\frac q2}n^{-\frac q2}v^{-\frac q2}\E^{\frac12}\big(\im R^{(\mathbb J,j)}_{kk})^{q}+C^qq^qn^{-\frac q2(1-2\alpha)}n^{-1}H_0^{q}s^q\notag\\&+
C^qq^qn^{-\frac q2(1-2\alpha)}n^{-1}\E^{\frac12}\big(\sum_{l\in\mathbb T_{\mathbb J,j,k}}|R^{(\mathbb J,k)}_{lk}|^{2q}\big)\Big).
\end{align}
Note that
\begin{equation}
|R^{(\mathbb J,j)}_{lk}(u+isv)-R^{(\mathbb J,j)}_{lk}(u+iv)|\le (s-1)v|[\mathbf R^{(\mathbb J,j)}(u+isv)\mathbf R^{(\mathbb J,j)}(u+isv)]_{lk}|.
\end{equation}
Applying the Cauchy --Schwartz inequality and \cite[Lemma ]{GT:2015}, we  get
\begin{equation}
|R^{(\mathbb J,j)}_{lk}(u+isv)-R^{(\mathbb J,j)}_{lk}(u+iv)|\le \sqrt s\sqrt{|R^{(\mathbb J,j)}_{ll}(u+isv)||R^{(\mathbb J,j)}_{kk}(u+iv)|}.
\end{equation}
Using now condition \eqref{h0}, we obtain
\begin{equation}\label{13}
\E|R^{(\mathbb J,j)}_{lk}(u+iv)|^q\le 2^q\E|R^{(\mathbb J,j)}_{lk}(u+isv)|^q+ 2^q(sH_0)^q\le 2^{q+1}s^q(sH_0)^q.
\end{equation}
Inequalities \eqref{11} and \eqref{13} together imply
\begin{align}
\E|R_{jk}^{(\mathbb J)}|^q\le H_0^q\left(\left(\frac{Cqs^3H_0}{nv}\right)^{\frac q2}+\frac1{\sqrt n}\left(\frac{4Cqs^2H_0}{n^{\frac{1-2\alpha}2}}\right)^q\right).
\end{align}
We may choose the constant $A_0$ such that for sufficiently large $n$
\begin{align}
\E|R_{jk}^{(\mathbb J)}|^q\le H_0^q.
\end{align}
 This completes the proof of Lemma \ref{e5}.
\end{proof}

\begin{lem}
\label{e2}Assuming condition \eqref{h0} for $v\ge v_1$, 
 we have, for all $v\ge v_0$ and $v\ge v_1/s$,
\begin{equation}
\E|\varepsilon_{j2}|^{2p}\le \left(\frac{Cp^4s^2H_0^2}{(nv)^{2(1-2\alpha)}}\right)^{p}.
\end{equation}
\end{lem}
\begin{proof}
Recall that
\begin{equation}
\varepsilon_{j2}=\frac1n\sum_{l\ne k\in\mathbb T_j}X_{jl}X_{jk}R^{(j)}_{kl}.
\end{equation}
Applying Lemma \ref{Gine}, we obtain
\begin{align}\label{new2}
\E|\varepsilon_{j2}|^{2p}&\le C^pn^{-2p}\Big(p^{2p}\E(\sum_{l\ne k\in \mathbb T_j}|R^{(j)}_{lk}|^2)^p+p^{3p}\mu_{2p}\sum_{k\in\mathbb T_j}\E(\sum_{l\in\mathbb T_{jk}}|R_{kl}^{(j)}|^2)^{p}\notag\\&+p^{4p}\mu_{2p}^2\sum_{l\ne k\in\mathbb T_j}\E|R_{kl}^{(j)}|^{2p}\Big).
\end{align}
At first we apply  \cite{GT:2015}[Lemma 7.6 inequalities (7.11), (7,12)] and condition \eqref{h0} of Lemma \ref{e5} obtaining
\begin{align}\label{improve}
\E|\varepsilon_{j2}|^{2p}&\le
C^pn^{-2p}(p^pn^pv^{-p}s^pH_0^p\notag \\
&\qquad+p^{3p}\mu_{2p}nv^{-p}(s_0H_0)^p+p^{4p}\mu_{2p}^2\sum_{l\ne k\in\mathbb T_j}\E|R_{kl}^{(j)}|^{2p}).
\end{align} 
Using the assumptions of Theorem \ref{trun} we have 
\begin{equation}\label{mom4}
\mu_{2p}\le D^{2p}n^{2p\alpha}n^{-2}\mu_{4+\varkappa}.
\end{equation}
Combining the last two inequalities and using that $2p(1-\alpha)\ge p$, we get
\begin{align}\label{new1}
\E|\varepsilon_{j2}|^{2p}&\le\left(\frac{Cp^2sH_0}{nv}\right)^p
+\left(\frac{Cp^3H_0s}{nv}\right)^p
+\left(\frac{Cp^4s^2H_0^2}{n^{2(1-2\alpha)}}\right)^p.
\end{align}
 
Using that  $1/v\ge \frac14$ and $2p(1-2\alpha)\le p$ for $z\in\mathbb G$, we may write
\begin{align}
\E|\varepsilon_{j2}|^{2p}\le
\left(\frac{Cp^4s^2H_0^2}{(nv)^{2(1-2\alpha)}}\right)^{p}.
\end{align}
Thus Lemma \ref{e2} is proved.
 \end{proof}
 \begin{lem}\label{rjj2}
 Assume that for some $1\ge v_1\ge v_0$ there exists a sufficiently large constant $H_0$ such that for any $\mathbb J$ with cardinality $|\mathbb J|\le \log_s(nv)$
and $q\le A_1(nv)^{\frac{1-2\alpha}2}$ the inequality
\begin{equation}\label{h01}
\E|R^{(\mathbb J)}_{jk}|^q\le H_0^q
\end{equation}
holds for any $j,k\in\mathbb T\setminus\mathbb J$ and any $v\ge v_1$.
Then  
\begin{equation}\label{h02}
\E|R^{(\mathbb J)}_{jj}|^q\le H_0^q
\end{equation} 
 holds for any $v\ge v_1/s$ and $j\in\mathbb T\setminus\mathbb J$.
\end{lem}
\begin{proof}First we use representation (3.7) in \cite{GT:2015}. We have
\begin{equation}\label{rjj1}
\R_{jj}^{(\mathbb J)}=s(z)+s(z)(\varepsilon^{(\mathbb J)}+\Lambda_n^{(\mathbb J)})R^{(\mathbb J)}_{jj},
\end{equation}
where $\varepsilon^{(\mathbb J)}_j=\varepsilon^{(\mathbb J)}_{j1}+\cdots+\varepsilon^{(\mathbb J)}_{j4}$ and
\begin{align}\label{defeps}
\varepsilon^{(\mathbb J)}_{j1}&=\frac1{\sqrt n}X_{jj},\quad \varepsilon^{(\mathbb J)}_{j2}=-\frac1n\sum_{l\ne k\in\mathbb T_{\mathbb J}}X_{jl}X_{jk}R^{\mathbb J,j)}_{lk},\notag\\
\varepsilon^{(\mathbb J)}_{j3}&=-\frac1n\sum_{l\in\mathbb T_{\mathbb J}}(X_{jl}^2-1)R^{(\mathbb J,j)}_{ll},\quad 
\varepsilon^{(\mathbb J)}_{j4}=\frac1n(\Tr \mathbf R^{(\mathbb J)}-\Tr\mathbf R^{\mathbb J,j)}).
\end{align}
Equality \eqref{rjj1} yields 
\begin{equation}\label{newrjj}
|R_{jj}^{(\mathbb J)}|\le1+  C(|\varepsilon_{j}^{(\mathbb J)}|+|\Lambda_n^{(\mathbb J)}(z)|)|R^{(\mathbb J)}_{jj}|.
\end{equation}

Since $|\Lambda_n|\le C\sqrt{|T_n|}$ for $z\in\mathbb G$ (see \cite[Lemma 5.9]{GT:2015}), we get
\begin{equation}
|R_{jj}^{(\mathbb J)}|\le 1+C|R_{jj}^{(\mathbb J)}|(|\varepsilon_j^{(\mathbb J)}|+\sqrt{|T_n|}).
\end{equation}
Applying Cauchy's inequality, we get
\begin{equation}
\E|R_{jj}^{(\mathbb J)}|^p\le 3^p(1+(\E^{\frac12}|\varepsilon_j^{(\mathbb J)}|^{2p}+\E^{\frac12}|T_n^{(\mathbb J)}|^{p})\E^{\frac12}|R_{jj}^{(\mathbb J)}|^{2p}).
\end{equation}
Note that if $p\le A_1(nv/s)^{\frac{1-2\alpha}2}$, then $2p\le A_1(nv)^{\frac{1-2\alpha}2}$.
Applying our assumption we get
\begin{equation}\label{2.10}
\E|R_{jj}^{(\mathbb J)}|^p\le 3^p+(sH_0)^p(\E^{\frac12}|\varepsilon_j^{(\mathbb J)}|^{2p}+\E^{\frac12}|T_n^{(\mathbb J)}|^{p})
\end{equation}Consider inequality
\begin{equation}\label{2.10a}
\E|R_{jj}^{(\mathbb J)}|^p\le 3^p+3^p(sH_0)^p(\E^{\frac12}|\varepsilon_j^{(\mathbb J)}|^{2p}+\E^{\frac12}|T_n^{(\mathbb J)}|^{p})
\end{equation}
By definition of $T_n$ and Cauchy's inequality, we have
\begin{equation}\label{2.10b}
\E^{\frac12}|T_n^{(\mathbb J)}|^{p}\le \Big(\frac1n\sum_{l\in\mathbb T_{\mathbb J}}\E^{\frac12}|\varepsilon_l^{(\mathbb J)}|^{2p}\E^{\frac12}|R^{(\mathbb J)}_{ll}|^{2p}\Big)^{\frac12}.
\end{equation}
Applying the inequality $|\varepsilon_l^{(\mathbb J)}|\le |\varepsilon_{l1}^{(\mathbb J)}|+|\varepsilon_{l2}^{(\mathbb J)}|+|\varepsilon_{l3}^{(\mathbb J)}|+|\varepsilon_{l4}^{(\mathbb J)}|$ and Lemmas \ref{e1}, \ref{e2}, \ref{e3} and \cite[Lemma 7.12]{GT:2015}, we get
\begin{align}\label{2.10c}
\E^{\frac12}|T_n^{(\mathbb J)}|^{p}\le C^p(sH_0)^{\frac p2} \Bigg(
\left(\frac{p^4s^2H_0^2}{(nv)^{2(1-2\alpha)}}\right)^{\frac p2}
&+\left(\frac{C}{n^{(1-2\alpha)}}\right)^{\frac p2}\Bigg).
\end{align}
We use here that $1\le 4v^{-1}$ and $p\ge 2p(1-2\alpha)$.
Similarly we get
\begin{align}\label{2.10d}
\E^{\frac12}|\varepsilon_j^{(\mathbb J)}|^{2p}\le
\left(\frac{p^4s^2H_0^2}{(nv)^{2(1-2\alpha)}}\right)^{p}
&+\left(\frac{C}{n^{(1-2\alpha)}}\right)^{p}.
\end{align}
Combining inequalities \eqref{2.10a} -- \eqref{2.10d}, we arrive at
\begin{align}
\E|R_{jj}^{(\mathbb J)}|^p\le 3^p&+\Bigg(\left(\frac{9C^2p^4(sH_0)^5}{(nv)^{2(1-2\alpha)}}\right)^{\frac p2}+
\left(\frac{9C^2s^3H_0^3}{n^{(1-2\alpha)}}\right)^{\frac p2}+\left(\frac{3Cp^4(sH_0)^3}{(nv)^{2(1-2\alpha)}}\right)^{p}\notag\\&+
\left(\frac{3CsH_0}{n^{1-2\alpha}}\right)^{p}\Bigg).
\end{align}
This inequality ensures that we may choose the constants $A_1$ and $A_0$ such that, for $v\ge v_1/s$
\begin{equation}
\E|R^{(\mathbb J)}_{jj}|^p\le H_0^p.
\end{equation}
 Thus Lemma \ref{rjj2} is proved.
 \end{proof}
To prove Theorem \ref{rjj} it is enough to repeat the proof of Lemma 5.13 and Corollary 5.14  in \cite{GT:2015}.

\section{{Proof} of Theorem \ref{stieltjesmain}}\label{expect}
We return now to the representation \eqref{rjj1} which implies that
\begin{align}\label{lambda}
s_n(z)&=\frac1n\sum_{j=1}^n\E R_{jj}=s(z)+\E\Lambda_n=s(z)+\E\frac{T_n(z)}{z+s(z)+m_n(z)}.
\end{align}
The last equality may be further reformulated as
\begin{align}\label{eq00}
s_n(z)=s(z)+\E\frac{\frac1n\sum_{j=1}^n\varepsilon_{j4}R_{jj}}{z+s(z)+m_n(z)}+
\E\frac{\widehat T_n(z)}{z+s(z)+m_n(z)},
\end{align}
where 
$$
\widehat T_n=\sum_{\nu=1}^3\frac1n\sum_{j=1}^n\varepsilon_{j\nu}R_{jj}.
$$
Note that the definition of $\varepsilon_{j4}$ in \eqref{rjj1} ($\mathbb J=\emptyset$) and equality 
\begin{equation}\label{shur}
 \Tr \mathbf R-\Tr\mathbf R^{(j)}=(1+\frac1n\sum_{l,k\in\mathbb T_j}X_{jl}X_{jk}[(R^{(j)})^2]_{kl})R_{jj}=R^{-1}_{jj}\frac {dR_{jj}}{dz},
\end{equation}
(see as well \cite[equality (7.34)]{GT:2015}) together imply 
\begin{equation}\label{7.3}
\frac1n\sum_{j=1}^n\varepsilon_{j4}R_{jj}=\frac1n\Tr \mathbf R^2=\frac1n\frac{d m_n(z)}{dz}.
\end{equation}
Thus we may rewrite \eqref{eq00} as
\begin{align}\label{eq01}
s_n(z)&=s(z)+\frac1n\E\frac{ m_n'(z)}{z+s(z)+m_n(z)}+
\E\frac{\widehat T_n(z)}{z+s(z)+m_n(z)}.
\end{align}

Denote
\begin{equation}\mathfrak T=\E\frac{\widehat T_n(z)}{z+s(z)+m_n(z)}.
\end{equation}

\subsection{Estimation of  $\mathfrak T$} We represent $\mathfrak T$ as follows
\begin{align}\notag
\mathfrak T=\mathfrak T_{1}+\mathfrak T_{2},
\end{align}
where
\begin{align}
\mathfrak T_{1}&=-\frac1n\sum_{j=1}^n\sum_{\nu=1}^3\E\frac{\varepsilon_{j\nu}
\frac1{z+m_n^{(j)}(z)}}{z+m_n(z)+s(z)},\notag\\
\mathfrak T_{2}&=\frac1n\sum_{j=1}^n\sum_{\nu=1}^3
\E\frac{\varepsilon_{j\nu}(R_{jj}+\frac1{z+m_n^{(j)}(z)})}{z+m_n(z)+s(z)}.\notag
\end{align}
\subsubsection{Estimation  of $\mathfrak T_{1}$}
We may decompose  $\mathfrak T_{1}$ as
\begin{align}\label{tii}
\mathfrak T_{1}=\mathfrak T_{11}+\mathfrak T_{12},
\end{align} 
where
\begin{align}
\mathfrak T_{11}&=-\frac1n\sum_{j=1}^n\sum_{\nu=1}^3\E\frac{\varepsilon_{j\nu}
\frac1{z+m_n^{(j)}(z)}}{z+m_n^{(j)}(z)+s(z)},\notag\\
\mathfrak T_{12}&=-\frac1n\sum_{j=1}^n\sum_{\nu=1}^3\E\frac{\varepsilon_{j\nu}\varepsilon_{j4}
\frac1{z+m_n^{(j)}(z)}}{(z+m_n^{(j)}(z)+s(z))(z+m_n(z)+s(z))}.\notag
\end{align}

It is easy to see that, by conditional expectation
\begin{equation}\label{fin104}
\mathfrak T_{11}=0.
\end{equation}
Applying the Cauchy--Schwartz inequality, for $\nu=1,2,3$,  we get
\begin{align}\label{fin100}
\Bigg|\E&\frac{\varepsilon_{j\nu}\varepsilon_{j4}
\frac1{z+m_n^{(j)}(z)}}{(z+m_n^{(j)}(z)+s(z))(z+m_n(z)+s(z))}\Bigg|\notag\\
&\le 
\E^{\frac12}\Bigg|\frac{\varepsilon_{j\nu}}{(z+m_n^{(j)}(z))(z+m_n^{(j)}(z)+s(z))}\Bigg|^2\E^{\frac12}\Bigg|\frac{\varepsilon_{j4}}{z+m_n(z)+s(z)}\Bigg|^2.
\end{align}
Conditioning and using \cite[Lemmas 7.15, 7.16]{GT:2015} together with the bound $\im m_n^{(j)}(z)\le|z+m_n^{(j)}(z)+s(z)|$, we get, for $\nu=2,3$,
\begin{align}
 \E^{\frac12}\Bigg|&\frac{\varepsilon_{j\nu}}{(z+m_n^{(j)}(z))(z+m_n^{(j)}(z)+s(z))}\Bigg|^2\le 
 \frac C{\sqrt{nv}}\E^{\frac12}\frac1{|z+m_n^{(j)}(z)|^2|z+m_n^{(j)}(z)+s(z)|}.\notag
\end{align}
Furthermore, applying \cite[Lemma 7.5]{GT:2015}, we get
\begin{align}\label{fin200}
 \E^{\frac12}\Bigg|&\frac{\varepsilon_{j\nu}}{(z+m_n^{(j)}(z))(z+m_n^{(j)}(z)+s(z))}\Bigg|^2\le 
 \frac C{\sqrt{nv}|z^2-4|^{\frac14}}\E^{\frac12}\frac1{|z+m_n^{(j)}(z)|^2}.
\end{align}
Inequalities \eqref{fin100}, \eqref{fin200},   Theorem \ref{rjj} together imply, for $\nu=2,3$,
\begin{align}\label{f19}
 \Bigg|\E&\frac{\varepsilon_{j\nu}\varepsilon_{j4}
\frac1{z+m_n^{(j)}(z)}}{(z+m_n^{(j)}(z)+s(z))(z+m_n(z)+s(z))}\Bigg|
&\le \frac C{\sqrt{nv}|z^2-4|^{\frac14}}\E^{\frac12}\frac{|\varepsilon_{j4}|^2}{|z+m_n^{(j)}(z)+s(z)|^2}.
\end{align}
By \cite[Lemma 7.5]{GT:2015} we have for $\nu=1$ 
\begin{align}\label{f20}
 \E^{\frac12}\frac{|\varepsilon_{j\nu}|^2}{|z+m_n^{(j)}(z)+s(z)|^2|z+m_n^{(j)}(z)|^2}&\le \frac {C}{\sqrt n\sqrt{|z^2-4|}} \E^{\frac12}\frac1{|z+m_n^{(j)}(z)|^2}\notag\\&\le  \frac {C}{\sqrt {nv}{|z^2-4|^{\frac14}}} \E^{\frac12}\frac1{|z+m_n^{(j)}(z)|^2}.
\end{align}
Applying Theorem \ref{rjj}, we get
\begin{align}\label{f20a}
 \E^{\frac12}\frac{|\varepsilon_{j\nu}|^2}{|z+m_n^{(j)}(z)+s(z)|^2|z+m_n^{(j)}(z)|^2}\le  \frac {C}{\sqrt {nv}{|z^2-4|^{\frac14}}}.
\end{align}
Furthermore, according to  Lemma \ref{modify}, we have, for $z\in\mathbb G$,
\begin{equation}
\E^{\frac12}\Bigg|\frac{\varepsilon_{j4}}{z+m_n(z)+s(z)}\Bigg|^2\le \frac C{nv}.
\end{equation}
Finally we get
\begin{align}
\Big|\mathcal T_{12}\Big|\le 
\frac C{(nv)^{\frac32}|z^2-4|^{\frac14}}.
\end{align}
\subsubsection{Estimation of $\mathfrak T_{2}$}
Using the representation \eqref{rjj1}, we write
\begin{align}\notag
\mathfrak T_{2}=\frac1n\sum_{j=1}^n\E\frac{\widetilde{\varepsilon}_j^2R_{jj}}{(z+m^{(j)}(z))(z+s(z)+m_n(z))}.
\end{align}
Furthermore we note that
\begin{equation}\notag
\widetilde{\varepsilon}_j^2=(\varepsilon_{j1}+\varepsilon_{j2}+\varepsilon_{j3})^2=\varepsilon_{j2}^2+\eta_j,
\end{equation}
where
\begin{equation}\notag
\eta_j=(\varepsilon_{j1}+\varepsilon_{j3})^2+2(\varepsilon_{j1}+\varepsilon_{j3})
\varepsilon_{j2}.
\end{equation}

We  now decompose $\mathfrak T_{2}$ as follows
\begin{equation}\label{t2ii}
\mathfrak T_{2}=\mathfrak T_{21}+\mathfrak T_{22}+\mathfrak T_{23},
\end{equation}
where
\begin{align}
\mathfrak T_{21}&=\frac1n\sum_{j=1}^n\E\frac{{\varepsilon}_{j2}^2R_{jj}}{(z+m^{(j)}(z))(z+s(z)+m_n(z))},\notag\\
\mathfrak T_{22}&=\frac1n\sum_{j=1}^n\E\frac{\eta_jR_{jj}}{(z+m^{(j)}(z))(z+s(z)+m_n^{(j)}(z))},\notag\\
\mathfrak T_{23}&=\frac1n\sum_{j=1}^n\E\frac{\eta_jR_{jj}\varepsilon_{j4}}{(z+m^{(j)}(z))(z+s(z)+m_n^{(j)}(z))(z+s(z)+m_n(z))}.\notag
\end{align}
First we note that
\begin{equation}
|\eta_j|\le 2|\varepsilon_{j1}|^2+2|\varepsilon_{j3}|^2+2|\varepsilon_{j2}|(|\varepsilon_{j1}|+|\varepsilon_{j2}|).
\end{equation}
Applying the Cauchy -- Schwartz inequality, we obtain
\begin{align}\label{lang2}
 \frac1n\sum_{j=1}^n\E\frac{(2|\varepsilon_{j1}|^2+2|\varepsilon_{j3}|^2)|R_{jj}|}{|z+m^{(j)}(z)||z+s(z)+m_n^{(j)}(z)|}&\le 
\frac Cn\sum_{j=1}^n\E^{\frac4{4+\varkappa}}\frac{|\varepsilon_{j1}|^{\frac{4+\varkappa}2}+|\varepsilon_{j3}|^{\frac{4+\varkappa}2}}{|z+s(z)+m_n^{(j)}(z)|^{\frac{4+\varkappa}{4}}}\notag\\&\times
\E^{\frac{\varkappa}{(4+\varkappa}}|R_{jj}|^{\frac{4+\varkappa)}{\varkappa}}|z+m^{(j)}(z)|^{-\frac{(4+\varkappa)}{\varkappa}} .
\end{align}
Lemma \ref{eps3} and \cite[Lemma 7.5 inequality (7.9)]{GT:2015} together imply
\begin{align}\label{in1}
\frac1n\sum_{j=1}^n\E\frac{(2|\varepsilon_{j1}|^2+2|\varepsilon_{j3}|^2)|R_{jj}|}{|z+m^{(j)}(z)||z+s(z)+m_n^{(j)}(z)|}&\le 
\frac C{n|z^2-4|^{\frac12}}.
\end{align}
Furthermore, applying H\"older's inequality, we get
\begin{align}
\frac1n\sum_{j=1}^n&\E\frac{2(|\varepsilon_{j1}|+|\varepsilon_{j3}|)|\varepsilon_{j2}||R_{jj}|}{|z+m^{(j)}(z)||z+s(z)+m_n^{(j)}(z)|}
\le \frac Cn\sum_{j=1}^n\E^{\frac2{4+\varkappa}}
\frac{|\varepsilon_{j1}|^{\frac{4+\varkappa}2}
+|\varepsilon_{j3}|^{\frac{4+\varkappa}2}}
{|z+s(z)+m_n^{(j)}(z)|^{\frac{4+\varkappa}{4}}}
\notag\\&\times \E^{\frac2{4+\varkappa}}\frac{|\varepsilon_{j2}|^{\frac{4+\varkappa}2}}{|z+s(z)+m_n^{(j)}(z)|^{\frac{4+\varkappa}{4}}}
\E^{\frac{\varkappa}{(4+\varkappa}}|R_{jj}|^{\frac{4+\varkappa)}{\varkappa}}|z+m^{(j)}(z)|^{-\frac{(4+\varkappa)}{\varkappa}}.
\end{align}
Applying Lemmas \ref{eps2}, \ref{eps3} and \cite[Lemma 7.5 inequality (7.9)]{GT:2015}, we obtain
\begin{equation}\label{in2}
\frac1n\sum_{j=1}^n\E\frac{(2(|\varepsilon_{j1}|+|\varepsilon_{j3}|)|\varepsilon_{j2}|)|R_{jj}|}{|z+m^{(j)}(z)||z+s(z)+m_n^{(j)}(z)|}
\le \frac C{n\sqrt v|z^2-4|^{\frac14}}.
\end{equation}
Combining inequalities \eqref{in1} and \eqref{in2}, we get
\begin{equation}\label{finisch2}
\mathcal T_{22}\le \frac C{nv^{\frac34}}.
\end{equation}
Applying \cite[inequality (7.42)]{GT:2015}, we obtain
\begin{align}
|\mathfrak T_{23}|&\le \frac Cn\sum_{j=1}^n\E\frac{|\eta_j||R_{jj}|}{|z+m^{(j)}(z)||z+s(z)+m_n^{(j)}(z)|}
\end{align}
Repeating the arguments of inequality \eqref{finisch2}, we similarly get
\begin{equation}\label{finisch3}
 |\mathfrak T_{23}|\le \frac C{nv^{\frac34}}.
\end{equation}

We continue now with $\mathfrak T_{21}$. We represent it in the form
\begin{equation}\label{t21ii}
\mathfrak T_{21}=H_1+H_2,
\end{equation}
where
\begin{align}
H_1&=-\frac1n\sum_{j=1}^n\E\frac{{\varepsilon}_{j2}^2}{(z+m^{(j)}(z))^2(z+s(z)+m_n(z))},\notag\\
H_2&=\frac1n\sum_{j=1}^n\E\frac{{\varepsilon}_{j2}^2(R_{jj}+\frac1{z+m_n^{(j)}})}{(z+m^{(j)}(z))(z+s(z)+m_n(z))}.\notag
\end{align}
Furthermore, using the representation
\begin{equation}
 R_{jj}=-\frac1{z+m_n^{(j)}(z)}+\frac1{z+m_n^{(j)}(z)}(\varepsilon_{j1}+\varepsilon_{j2}+\varepsilon_{j3})R_{jj},
\end{equation}
we bound $H_2$ in the following way
\begin{align}\notag
|H_2|\le H_{21}+H_{22}+H_{23},
\end{align}
where
\begin{align}
H_{21}&=\frac1n\sum_{j=1}^n\E\frac{4|{\varepsilon}_{j1}|^3|R_{jj}|}{|z+m^{(j)}(z)|^2|z+s(z)+m_n(z)|},\notag\\
H_{22}&=\frac1n\sum_{j=1}^n\E\frac{2|{\varepsilon}_{j2}|^3|R_{jj}|}{|z+m^{(j)}(z)|^2|z+s(z)+m_n(z)|},\notag\\
H_{23}&=\frac1n\sum_{j=1}^n\E\frac{2|{\varepsilon}_{j3}||\varepsilon_{j2}|^2|R_{jj}|}{|z+m^{(j)}(z)|^2|z+s(z)+m_n(z)|}.\notag
\end{align}
Using \cite[inequality (7.42)]{GT:2015}, we get, for $\nu=1,2$
\begin{align}\label{inr3}
H_{2\nu}&\le \frac1n\sum_{j=1}^n\E\frac{4|{\varepsilon}_{j\nu}|^3|R_{jj}|}{|z+m^{(j)}(z)|^2|z+s(z)+m_n^{(j)}(z)|}\notag\\&
+\frac1n\sum_{j=1}^n\E\frac{4|{\varepsilon}_{j\nu}|^3|R_{jj}||\varepsilon_{j4}|}{|z+m^{(j)}(z)|^2|z+s(z)+m_n^{(j)}(z)||z+s(z)+m_n(z)|}\notag\\&
\le \frac Cn\sum_{j=1}^n\E^{\frac34}\frac{|{\varepsilon}_{j\nu}|^4}{|z+m_n^{(j)}(z)|^{\frac83}|z+s(z)+m_n^{(j)}(z)|^{\frac43}}
\E^{\frac14}|R_{jj}|^4.
\end{align}
Applying \cite[Corollary 7.17 inequality (7.32)]{GT:2015}  with $\beta=\frac43$ and $\alpha=\frac83$, we obtain,  for $z\in\mathbb G$, and for $\nu=1,2$
\begin{align}\label{h2nu}
\E^{\frac34}\frac{|{\varepsilon}_{j\nu}|^4}{|z+m^{(j)}(z)|^{\frac83}|z+s(z)+m_n^{(j)}(z)|^{\frac43}}\le
\frac C{(nv)^{\frac32}}.\notag
\end{align}
Furthermore, using H\"older's inequality, we get
\begin{align}
H_{23}\le \frac1n\sum_{j=1}^n\E^{\frac2{4+\varkappa}}\frac{|\varepsilon_{j2}|^{4+\varkappa}}{|z+s(z)+m_n^{(j)}(z)|^{\frac{4+\varkappa}4}}
\E^{\frac2{4+\varkappa}}|\varepsilon_{j3}|^{\frac{4+\varkappa}2}\E^{\frac{\varkappa}{4+\varkappa}}\left(\frac{|R_{jj}|}{|z+m^{(j)}(z)|}\right)^{\frac{4+\varkappa}{\varkappa}}.
\end{align} 
Applying Lemmas \ref{eps2}, \ref{eps3}, we get
\begin{equation}\label{h23}
H_{23}\le \frac C{n^{\frac32}v}\le
\frac C{(nv)^{\frac32}}.
\end{equation}
Combining inequalities \eqref{h2nu} and \eqref{h23}, we arrive at
\begin{equation}
H_{2}\le 
\frac C{(nv)^{\frac32}}.
\end{equation}
Consider now $H_1$. Using the equality
\begin{equation}\notag
 \frac1{z+m_n(z)+s(z)}=\frac1{z+2s(z)}-\frac{\Lambda_n(z)}{(z+2s(z))(z+m_n(z)+s(z))}
\end{equation}
and
\begin{equation}\label{lepsj4}
 \Lambda_n=\Lambda_n^{(j)}+\varepsilon_{j4},
\end{equation}
we represent it in the form
\begin{align}\label{h1}
H_1=H_{11}+H_{12}+H_{13},
\end{align}
where
\begin{align}
H_{11}&=-\frac1{(z+s(z))^2}\frac1n\sum_{j=1}^n\E\frac{\varepsilon_{j2}^2}{z+s(z)+m_n(z)}\notag\\&
=-s^2(z)\frac1n\sum_{j=1}^n\E\frac{\varepsilon_{j2}^2}{z+s(z)+m_n(z)},\notag\\
H_{12}&=-\frac1{(z+s(z))}\frac1n\sum_{j=1}^n\E\frac{\varepsilon_{j2}^2\Lambda_n^{(j)}
}{(z+m_n^{(j)}(z))^2(z+s(z)+m_n(z))},\notag\\
H_{13}&=-\frac1{(z+s(z))^2}\frac1n\sum_{j=1}^n\E\frac{\varepsilon_{j2}^2
\Lambda_n^{(j)}
}{(z+m_n^{(j)}(z))(z+s(z)+m_n(z))}.\notag
\end{align}
In order to apply conditional independence, we write 
\begin{align}\notag
H_{11}=H_{111}+H_{112},
\end{align}
where
\begin{align}
H_{111}&=-s^2(z)\frac1n\sum_{j=1}^n\E\frac{\varepsilon_{j2}^2}{z+m_n^{(j)}(z)+s(z)},\notag\\
H_{112}&=s^2(z)\frac1n\sum_{j=1}^n\E\frac{\varepsilon_{j2}^2\varepsilon_{j4}}{(z+s(z)+m_n(z))(z+m_n^{(j)}(z)+s(z))}.\notag
\end{align}
It is straightforward to check that
\begin{align}\notag
\E\{\varepsilon_{j2}^2|\mathfrak M^{(j)}\}=\frac1{n^2}\Tr (\mathbf R^{(j)})^2-\frac1{n^2}\sum_{l\in\mathbb T_j}(R^{(j)}_{ll})^2.
\end{align}
Using equality \eqref{7.3} for $m_n'(z)$ and the corresponding relation for ${m_n^{(j)}}'(z)$, we may write
\begin{equation}\notag
H_{111}=L_1+L_2+L_3+L_4,
\end{equation}
where
\begin{align}
L_1&=-s^2(z)\frac1n\E\frac{m_n'(z)}{z+m_n(z)+s(z)},\notag\\
L_2&=s^2(z)\frac1n\sum_{j=1}^n\E\frac{\frac1{n^2}\sum_{l\in\mathbb T_j}
(R^{(j)}_{ll})^2}{z+m_n^{(j)}(z)+s(z)},\\
L_3&=s^2(z)\frac1n\sum_{j=1}^n\E\frac{\frac1{n}((m_n^{(j)}(z))'-m_n'(z))}{z+m_n^{(j)}(z)+s(z)},\\
L_4&=s^2(z)\frac1n\sum_{j=1}^n\E\frac{\frac1{n}((m_n^{(j)}(z))'-m_n'(z))\varepsilon_{j4}}
{(z+m_n(z)+s(z))(z+m_n^{(j)}(z)+s(z))}.\notag
\end{align}
 Using \cite[Lemma 7.6 inequality (7.11)]{GT:2015} and Theorem \ref{rjj}, we get
 \begin{align}\label{l2}
|L_2|&\le \frac C{n\sqrt{|z^2-4|}}.
\end{align}
Note that
\begin{equation}
m_n^{(j)}(z))'-m_n'(z)=\int_{-\infty}^{\infty}\frac1{(x-z)^2}d(\mathcal F_n^{(j)}(x)-\mathcal F_n(x)).
\end{equation}
Integrating by parts we obtain
\begin{equation}
|m_n^{(j)}(z))'-m_n'(z)|\le \frac C{nv^2}.
\end{equation}
The last inequality and \cite[Lemma 7.5 inequality (7.9)]{GT:2015} together imply
\begin{align}\label{l3}
|L_3|\le \frac C{n^2v^2\sqrt{|z^2-4|}}.
\end{align}
Finally, using that $|\varepsilon_{j4}|/|z+s(z)+m_n(z)|\le C$ for $z\in\mathbb G$, we arrive at
\begin{align}\label{l4}
|L_4|&\le \frac C{n^2v^2\sqrt{|z^2-4|}}.
\end{align}
Conditioning on $\mathfrak M^{(j)}$ and applying Lemma \ref{eps2}, we get
\begin{equation}\label{h112}
|H_{112}|\le \frac C{n^2v^2|z^2-4|^{\frac12}}.
\end{equation}
Applying \cite[ inequality (7.42)]{GT:2015} , we may write
\begin{align}\notag
 |H_{12}|\le \frac Cn\sum_{j=1}^n \E\frac{|\varepsilon_{j2}|^2|\Lambda_n^{(j)}|}
 {|z+m_n^{(j)}(z)||z+m_n^{(j)}(z)+s(z)|}.
\end{align}
Conditioning on $\mathfrak M^{(j)}$ and applying Lemma \ref{eps2}, we get
\begin{align}\notag
 |H_{12}|\le \frac C{{nv}}\frac 1n\sum_{j=1}^n \E\frac{|\Lambda_n^{(j)}|}{|z+m_n^{(j)}(z)|}\le 
 \frac C{nv}\E^{\frac12}|\Lambda_n|^2+\frac C{n^2v^2}.
\end{align}
By Lemma \ref{lambdaopt}, we get
\begin{equation}\label{h12}
 |H_{12}|\le\frac C{n^2v^2}.
\end{equation}
Similarly we obtain
\begin{equation}\label{h13}
 |H_{13}|\le \frac C{n^2v^2}.
\end{equation}

Now we rewrite the equations \eqref{eq00}  and \eqref{eq01} as follows, using the remainder term 
$\mathfrak T_3$,
which can be  bounded by means of inequalities \eqref{h12}, \eqref{finisch2}, \eqref{l2},\eqref{l3}, \eqref{l4}, \eqref{h112} and \eqref{h13}
\begin{align}\label{final00}
\E\Lambda_n(z)=\E m_n(z)-s(z)=\frac{(1-s^2(z))}{n}\E\frac{m_n'(z)}{z+m_n(z)+s(z)}+\mathfrak T_3,
\end{align}
where
\begin{align}\notag
|\mathfrak T_3|\le \frac C{nv^{\frac34}}+\frac C{n^{\frac32}v^{\frac32}|z^2-4|^{\frac14}}.
\end{align}
Note that
\begin{equation}\notag
1-s^2(z)=s(z)\sqrt{z^2-4}.
\end{equation}
In \eqref{final00} it remains to estimate the quantity
\begin{equation}\notag
\mathfrak T_4=-\frac{s(z)\sqrt{z^2-4}}{n}\E\frac{m_n'(z)}{z+m_n(z)+s(z)}.
\end{equation}

\subsection{Estimation of $\mathfrak T_4$}\label{better}
Using that $\Lambda_n=m_n(z)-s(z)$ we rewrite $\mathfrak T_4$ as
\begin{equation}\notag
\mathfrak T_4=\mathfrak T_{41}+\mathfrak T_{42}+\mathfrak T_{43},\notag
\end{equation}
where
\begin{align}
\mathfrak T_{41}&=-\frac{s(z)s'(z)}{n},\notag\\
\mathfrak T_{42}&=\frac{s(z)\sqrt{z^2-4}}n\E\frac{m_n'(z)-s'(z)}{z+m_n(z)+s(z)},\notag\\
\mathfrak T_{43}&=\frac{s(z)}n\E\frac{(m_n'(z)-s'(z))\Lambda_n}{z+m_n(z)+s(z)}.\notag
\end{align}

\subsubsection{Estimation of  $\mathfrak T_{42}$}
First we investigate $m_n'(z)$. The following equality holds
\begin{equation}\label{mn'}
m_n'(z)=\frac1n\Tr R^2=\sum_{j=1}^n \varepsilon_{j4}R_{jj}=
s^2(z)\sum_{j=1}^n \varepsilon_{j4}R_{jj}^{-1}+D_1,
\end{equation}
where
\begin{equation}\label{d1}
D_1=\sum_{j=1}^n \varepsilon_{j4}(R_{jj}-s(z))(1+R_{jj}^{-1}s(z)).
\end{equation}
Using equality \eqref{shur}, we may write
\begin{equation}\notag
m_n'(z)=\frac{s^2(z)}n\sum_{j=1}^n (1+\frac1n\sum_{l,k\in\mathbb T_j}X_{jl}X_{jk}[(R^{(j)})^2]_{lk})+D_1.
\end{equation}
Denote
\begin{align}\beta_{j1}&=\frac1n\sum_{l\in\mathbb T_j}[(R^{(j)})^2]_{ll}-\frac1n\sum_{l=1}^n[(R)^2]_{ll}
=\frac1n\sum_{l\in\mathbb T_j}[(R^{(j)})^2]_{ll}-m_n'(z)\notag\\&=\frac1n\frac d{dz}(\Tr \mathbf R-\Tr\mathbf R^{(j)}),\notag\\
\beta_{j2}&=\frac1n\sum_{l\in\mathbb T_j}(X^2_{jl}-1)[(R^{(j)})^2]_{ll},\notag\\
\beta_{j3}&=\frac1n\sum_{l\ne k\in\mathbb T_j}X_{jl}X_{jk}[(R^{(j)})^2]_{lk}.
\end{align}
Using  these notations we may write
\begin{align}\notag
m_n'(z)=s^2(z)(1+m_n'(z))+\frac{s^2(z)}n\sum_{j=1}^n (\beta_{j1}+\beta_{j2}+\beta_{j3})+D_1.
\end{align}
Solving this equation with respect to $m_n'(z)$ we obtain
\begin{equation}\label{semi1}
m_n'(z)=\frac{s^2(z)}{1-s^2(z)}+\frac{1}{1-s^2(z)}(D_1+D_2),
\end{equation}
where
\begin{align}\notag
D_2=\frac{s^2(z)}n\sum_{j=1}^n (\beta_{j1}+\beta_{j2}+\beta_{j3}).
\end{align}
Note that for the semi-circular law the following identities hold
\begin{equation}\notag
\frac{s^2(z)}{1-s^2(z)}=\frac{s^2(z)}{1+\frac{s(z)}{z+s(z)}}=-\frac{s(z)}{z+2s(z)}=
s'(z).
\end{equation}
Applying this relation we rewrite equality \eqref{semi1} as
\begin{equation}\label{mnderiv}
m_n'(z)-s'(z)=\frac{1}{s(z)(z+2s(z))}(D_1+D_2).
\end{equation}
Using the last equality, we may represent $\mathfrak T_{42}$ now as follows
\begin{equation}\notag
\mathfrak T_{42}=\mathfrak T_{421}+\mathfrak T_{422},
\end{equation}
where
\begin{align}
\mathfrak T_{421}&=\frac1n\E \frac{D_1}{z+m_n(z)+s(z)},\notag\\
\mathfrak T_{422}&=\frac1n\E \frac{D_2}{z+m_n(z)+s(z)}.\notag
\end{align}



Recall that, by \eqref{d1},
\begin{align}\mathfrak T_{421}&=-\frac1{n}\sum_{j=1}^n
\E\frac{\varepsilon_{j4}(R_{jj}-s(z))(1+R_{jj}^{-1}s(z))}{(z+s(z)+m_n(z))}.
\end{align}
Using that $|\varepsilon_{j4}|\le 1/nv$ and
\begin{equation}
R_{jj}-s(z)=s(z)(\varepsilon_j+\Lambda_n)R_{jj}
\end{equation}
and that $|z+m_n(z)+s(z)|\ge c|z^2-4|^{\frac12}$ for $z\in\mathbb G$, it is straightforward to check that
\begin{align}\label{finisch7*}
|\mathfrak T_{421}|\le\frac C{n^{\frac32}v^{\frac32}|z^2-4|^{\frac14}}.
\end{align}
\subsubsection{Estimation of $\mathfrak T_{422}$}
We represent now $\mathfrak T_{422}$ in the form
\begin{align}\notag
\mathfrak T_{422}=\mathfrak T_{51}+\mathfrak T_{52}+\mathfrak T_{53},
\end{align}
where
\begin{align}\notag
\mathfrak T_{5\nu}&=\frac1{n^2}\sum_{j=1}^n\E\frac{\beta_{j\nu}}{z+m_n(z)+s(z)},\quad\text{for}\quad \nu=1,2,3.
\end{align}
At first we investigate $\mathfrak T_{53}$. Note that, by \cite[Lemma 7.26]{GT:2015},
\begin{equation}\notag
|\beta_{j1}|\le \frac C{nv^2}.
\end{equation}
Therefore, for $z\in\mathbb G$, using \cite[Lemma 7.5 inequality (7.9)]{GT:2015}, 
\begin{equation}\label{t51}
|\mathfrak T_{51}|\le \frac C{n^2v^2\sqrt{|z^2-4|}}\le \frac C{n^{\frac32}v^{\frac32}|z^2-4|^{\frac14}}.
\end{equation}

Furthermore, we consider the quantity $\mathfrak T_{5\nu}$, for $\nu=2,3$.
Applying the Cauchy-Schwartz inequality and \cite[inequality (7.42)]{GT:2015}  as well, 
we get
\begin{align}\notag
|\mathfrak T_{5\nu}|\le \frac C{n^2}\sum_{j=1}^n\E^{\frac12}\frac{|\beta_{j\nu}|^2}{|z+m_n^{(j)}(z)+s(z)|^2}.
\end{align}
By \cite[Lemma 7.25]{GT:2015} together with \cite[Lemma 7.5]{GT:2015}, we obtain
\begin{align}\notag
\E^{\frac12}\frac{|\beta_{j\nu}|^2}{|z+m_n^{(j)}(z)+s(z)|^2}\le \frac C{n^{\frac12}v^{\frac32}|z^2-4|^{\frac14}}.
\end{align}
This implies that
\begin{align}\label{finisch7}
|\mathfrak T_{5\nu}|\le \frac C{n^{\frac32}v^{\frac32}|z^2-4|^{\frac14}}.
\end{align}
Inequalities \eqref{t51} and \eqref{finisch7} yield
\begin{equation}\notag
 |\mathfrak T_{422}|\le \frac C{n^{\frac32}v^{\frac32}|z^2-4|^{\frac14}}.
\end{equation}
Combining \eqref{finisch7*} and \eqref{finisch7^}, we get, for $z\in\mathbb G$,
\begin{align}\label{finisch7^}
 |\mathfrak T_{42}|\le \frac C{n^{\frac32}v^{\frac32}|z^2-4|^{\frac14}}.
\end{align}
\subsubsection{Estimation of $\mathfrak T_{43}$} Recall that
\begin{equation}\notag
\mathfrak T_{43}=\frac{s(z)}n\E\frac{(m_n'(z)-s'(z))\Lambda_n}{z+m_n(z)+s(z)}.
\end{equation}
Applying equality \eqref{mnderiv}, we obtain
\begin{align}\notag
 \mathfrak T_{43}=\mathfrak T_{431}+\mathfrak T_{432},
\end{align}
where
\begin{align}
 \mathfrak T_{431}&=\frac{1}{n(z+2s(z))}\E\frac{D_1\Lambda_n}{z+m_n(z)+s(z)},\\
 \mathfrak T_{432}&=\frac{1}{n(z+2s(z))}\E\frac{D_2\Lambda_n}{z+m_n(z)+s(z)}.\notag
\end{align}

By definition of $D_1$,  we get
\begin{align}\notag
|\mathfrak T_{431}|\le \mathfrak T_{61}+\mathfrak T_{62},
\end{align}
where
\begin{align}
\mathfrak T_{61}&=\frac C{nv|z^2-4|^{\frac12}}\frac1n\sum_{j=1}^n\E|\varepsilon_j|(|R_{jj}|+1)|\Lambda_n|,\notag\\
\mathfrak T_{62}&=\frac C{nv}\frac1n\sum_{j=1}^n\E\frac{(|R_{jj}|+1)|\Lambda_n|^2}{|z+m_n(z)+s(z)|}.
\end{align}
Applying the Cauchy-Schwartz  inequality and Lemma \ref{lambdaopt}, we get
\begin{equation}
\mathfrak T_{61}\le \frac C{n^2v^2|z^2-4|^{\frac12}}.
\end{equation}
Furthermore, using that $|\lambda_n|\le \sqrt{|T_n|}$ for $z\in\mathbb G$, we get
\begin{equation}\label{t62}
\mathfrak T_{62}\le \frac C{nv}\frac1n\sum_{j=1}^n\E^{\frac12}(|R_{jj}|^2+1)\E^{\frac12}\frac{|T_n|^2}{|z+m_n(z)+s(z)|^2}.
\end{equation}
By definition of $T_n$ and H\"older's inequality, we have
\begin{equation}
\E\frac{|T_n|^2}{|z+m_n(z)+s(z)|^2}\le \frac1n\sum_{j=1}^n\E^{\frac4{4+\varkappa}}\frac{|\varepsilon_j|^{\frac{4+\varkappa}2}}{|z+m_n(z)+s(z)|^{\frac{4+\varkappa}2}}\E^{\frac{\varkappa}{4+\varkappa}}|R_{jj}|^{\frac{2(4+\varkappa)}{\varkappa}}.
\end{equation}
Applying now Lemmas \ref{eps2}, \ref{eps3} and Theorem \ref{rjj}, we get
\begin{equation}
\E\frac{|T_n|^2}{|z+m_n(z)+s(z)|^2}\le \frac {C}{n|z^2-4|}+\frac C{nv|z^2-4|^{\frac12}}.
\end{equation}
The last inequality together with inequality \eqref{t62}
implies, for $z\in\mathbb G$,
\begin{equation}
\mathfrak T_{62}\le \frac C{(nv)^{\frac32}|z^2-4|^{\frac 14}}.
\end{equation}

For $z\in\mathbb G$ we get
\begin{equation}\notag
|\mathfrak T_{431}|
 \le
\frac{4}{n^{\frac32} v^{\frac32}|z^2-4|^{\frac14}}.
\end{equation}
Applying  the Cauchy -- Schwartz inequality, we get for $\mathfrak T_{432}$ accordingly
\begin{align}\notag
 |\mathfrak T_{432}|\le \frac C{n|z^2-4|^{\frac12}}\E^{\frac12}|D_2|^2\E^{\frac12}|\Lambda_n|^2.
\end{align}
By Lemma \ref{lambdaopt}, we have
\begin{equation}\label{krak}
 |\mathfrak T_{432}|\le \frac C{n^2v|z^2-4|^{\frac12}}\E^{\frac12}|D_2|^2.
\end{equation}
By definition of $D_2$,
\begin{equation}\notag
 \E|D_2|^2\le \frac1n\sum_{j=1}^n(\E|\beta_{j1}|^2+\E|\beta_{j2}|^2+\E|\beta_{j3}|^2).
\end{equation}
Applying \cite[Lemma 7.26]{GT:2015}, \cite[Lemma 7.25]{GT:2015} with $\nu=2,3$,  we get
\begin{align}\label{krak1}
 \E|D_2|^2\le \frac{C}{n^2v^4}+\frac C{nv^3}.
\end{align}
Inequalities \eqref{krak} and \eqref{krak1} together imply, for $z\in\mathbb G$,
\begin{equation}\notag
 |\mathfrak T_{432}|\le \frac C{n^3v^3|z^2-4|^{\frac12}}+\frac C{n^{\frac52}v^{\frac52}|z^2-4|^{\frac12}}\le \frac C{n^{\frac32}v^{\frac32}|z^2-4|^{\frac14}}.
\end{equation}
Finally we observe that
\begin{equation}\notag
 s'(z)=-\frac{s(z)}{\sqrt{z^2-4}}
\end{equation}
and, therefore
\begin{equation}\notag
 |\mathfrak T_{41}|\le \frac C{n|z^2-4|^{\frac12}}. 
\end{equation}
For $z\in\mathbb G$ we may rewrite it 
\begin{equation}\label{final001}
 |\mathfrak T_{41}|\le \frac C{n\sqrt v}.
\end{equation}
Combining now relations \eqref{final00}, \eqref{h1},  \eqref{t51}, \eqref{finisch7^}, \eqref{final001}, we get for $z\in\mathbb G$,
\begin{equation}\notag
 |\E\Lambda_n|\le \frac C{n v^{\frac34}}+\frac C{n^{\frac32}v^{\frac32}|z^2-4|^{\frac14}}.
\end{equation}
The last inequality completes the proof of Theorem \ref{stieltjesmain}.
\section{Appendix}
We reformulate here a  result by Gine, Latala and Zinn (2000), \cite{Gine:2000}, for quadratic forms $Q:=\sum_{1\le l\ne k\le n}a_{lk}\xi_l\xi_k$.
\begin{lem}\label{Gine}
 Assume that  for  $q\ge 4$ and for any $j=1,\ldots n$
 $$
 \E|\xi_j|^q\le \mu_q.
 $$
 Then there exists an absolute constant $K>0$ such that, for $z\in\mathbb G$,
 \begin{align}
 \E|Q|^q\le K^p\Big(q^{q}\sigma^q\Big(\sum_{1\le l\ne k\le n}|a_{lk}|^2\Big)^{\frac q2}+\mu_qq^{\frac{3q}2}\sum_l(\sum_k|a_{kl}|^2)^{\frac q2}+q^{2 q}\mu_q^2\sum_{l,k}|a_{lk}|^q\Big).
 \end{align}
\end{lem}
\begin{proof}
For the proof see \cite[Section 3, inequality (3.3)]{Gine:2000} with $h_{ij}=a_{ij}\xi_i\xi_j$.
\end{proof}
\begin{lem}\label{modify} Assuming the conditions of Theorem \ref{trun},  we have, for any $\theta>0$
\begin{equation}\label{i0}
\E\left|\frac{\varepsilon_{j4}}{z+s(z)+m_n(z)}\right|^2\frac1{|z+m_n^{(j)}(z)|^{\theta}}\le \frac C{n^2v^2}
\end{equation}
with some positive constant $C$ depending on $\varkappa$, $A_{\varkappa}$ $D$ and $\theta$ .
\end{lem}
\begin{proof} Applying \cite[inequality (7.42)]{GT:2015}, we may write
\begin{align}
\E\left|\frac{\varepsilon_{j4}}{z+s(z)+m_n(z)}\right|^2 \frac 1{|z+m_n^{(j)}(z)|^{\theta} } \le \E\left|\frac{\varepsilon_{j4}}{z+s(z)+m^{(j)}_n(z)}\right|^2\frac 1 {|z+m_n^{(j)}(z)|^{\theta}}.
\end{align}
Furthermore, using  \cite[equality (7.34)]{GT:2015}, we get
\begin{align}
\E&\left|\frac{\varepsilon_{j4}}{z+s(z)+m_n^{(j)}(z)}\right|^2\frac 1{|z+m_n^{(j)}(z)|^{\theta}}\notag\\&=\E\left|\frac{(1+\frac1n\sum_{k\ne l\in\mathbb T_j}X_{jk}X_{jl}[\mathbf {R^{(j)}}^2]_{lk})R_{jj}}{z+s(z)+m_n^{(j)}(z)}\right|^2\frac1{|z+m_n^{(j)}(z)|^{\theta}}.
\end{align}
Introduce notations as in \cite{GT:2015}:
\begin{align}
\eta_{j0}&:=\frac1n(1+\frac1n\sum_{l\in\mathbb T_j}[\mathbf {R^{(j)}}^2]_{ll}),\quad\eta_{j1}:=\frac1{n^2}\sum_{l\in\mathbb T_j}(X_{jl}^2-1)[\mathbf {R^{(j)}}^2]_{ll},\notag\\
\eta_{j2}&:=\frac1{n^2}\sum_{l\ne k\in\mathbb T_j}X_{jl}X_{jk}[\mathbf {R^{(j)}}^2]_{lk}).
\end{align}
Using that
\begin{equation}
|\eta_{j0}|\le \frac1{nv}\im(z+m_n^{(j)}(z))\le \frac1{nv}\im(z+m_n^{(j)}(z)+s(z)),
\end{equation}
(see \cite[inequality (7.57)]{GT:2015})
we get
\begin{equation}
\E\left|\frac{\eta_{j0}}{z+s(z)+m_n^{(j)}(z)}\right|^2\frac1{|z+m_n^{(j)}(z)}|^{\theta}\le 
\frac1{(nv)^2}\E\frac1{|z+m_n^{(j)}(z)|^{\theta}}\le \frac C{(nv)^2}.
\end{equation}
Note that
\begin{equation}\label{i1}
\E\{|\eta_{j1}|^2\Big|\mathfrak M^{(j)}\}\le \frac 1{n^4}\sum_{l\in\mathbb T_j}|[\mathbf {R^{(j)}}^2]_{ll}^2\le \frac C{n^3v^2}
\end{equation}
and
\begin{equation}\label{i2}
\E\{|\eta_{j2}|^2\Big|\mathfrak M^{(j)}\}\le \frac 1{n^4}\sum_{l\ne k\in\mathbb T_j}|[\mathbf {R^{(j)}}^2]_{lk}^2\le \frac C{n^3v^3}\im m_n^{(j)}(z).
\end{equation}
See \cite[Lemma 6 inequalities (7.15), (7.16)]{GT:2015}.
Inequalities \eqref{i1}, \eqref{i2} together imply
\begin{align}
\E\left|\frac{\eta_{j1}+\eta_{j2}}{z+s(z)+m_n(z)}\right|^2&\frac1{|z+m_n^{(j)}(z)}|^{\theta}\notag\\&\le \frac C{n^3v^3}\E\left|\frac{v+\im m_n^{(j)}(z)}{z+s(z)+m_n(z)}\right|^2\frac1{|z+m_n^{(j)}(z)|^{\theta}}\notag\\&\le 
\frac C{n^3v^3|z^2-4|^{\frac12}}.
\end{align}
We use here \cite[Lemma 7.5 inequality (7.9)]{GT:2015}.
Furthermore, applying \cite[Lemma 7.13]{GT:2015}, we get, for $z\in \mathbb G$
\begin{equation}\label{i3}
\E\left|\frac{\eta_{j1}+\eta_{j2}}{z+s(z)+m_n(z)}\right|^2\frac1{|z+m_n^{(j)}(z)|^{\theta}}\le 
\frac C{n^2v^2}.
\end{equation}
Inequalities \eqref{i0} and \eqref{i1} conclude the proof of Lemma \ref{modify}. Thus Lemma \ref{modify} is proved.
\end{proof}
\begin{lem}\label{eps3}
Assuming the conditions of Theorem \ref{trun}  there exists a positive constant depending on $\varkappa, \mu_{4+\varkappa}$ and $D$ such that

\begin{equation}
\E^{\frac2{4+\varkappa}}\{|\varepsilon_{j3}|^{\frac{4+\varkappa}2}\Big|\mathfrak M^{(j)}\}\le \frac C{\sqrt n}(\frac1n\sum_{l=1}^n|R_{ll}|^2)^{\frac12}
\end{equation}
and
\begin{equation}
\E^{\frac2{4+\varkappa}}\{|\varepsilon_{j3}|^{\frac{4+\varkappa}2}\Big|\mathfrak M^{(j)}\}\le \frac C{\sqrt {nv}}(\im m_n^{(j)}(z))^{\frac12}
\end{equation}
\end{lem}
\begin{proof}
Applying Rosenthal's inequality, we get
\begin{align}
\E|\varepsilon_{j3}|^{\frac{4+\varkappa}2}\le \frac {C\mu_4}{n^{\frac{4+\varkappa}2}}(\sum_{l\in\mathbb T_j}|R^{(j)}_{ll}|^2)^{\frac{4+\varkappa}4}+
\frac C{n^{\frac{4+\varkappa}2}}\mu_{4+\varkappa}\sum_{l\in\mathbb T_j}|R^{(j)}_{ll}|^{\frac{4+\varkappa}2}.
\end{align}
Note that
\begin{equation}
\sum_{l\in\mathbb T_j}|R^{(j)}_{ll}|^{\frac{4+\varkappa}2}\le (\sum_{l\in\mathbb T_j}|R^{(j)}_{ll}|^2)^{\frac{4+\varkappa}4}.
\end{equation}
\end{proof}
These inequalities complete the proof of the first claim of Lemma \ref{eps3}. The second one is easy.
Thus Lemma \ref{eps3} is proved.
\begin{lem}\label{eps2}
Assuming the conditions of Theorem \ref{trun}  there exists a positive constant depending on $\varkappa, \mu_{4+\varkappa}$ and $D$ such that
\begin{equation}
\E^{\frac1{4+\varkappa}}\{|\varepsilon_{j2}|^{{4+\varkappa}}\Big|\mathfrak M^{(j)}\}\le \frac C{\sqrt {nv}}\im^{\frac12} m_n^{(j)}(z).
\end{equation}
\end{lem}

\begin{proof}Applying Lemma \ref{Gine}, we get
\begin{align}
\E\{|\varepsilon_{j2}|^{4+\varkappa}|\mathfrak M^{(j)}\}\le& Cn^{-(4+\varkappa)}\Bigg((\sum_{l\ne k\in\mathbb T_j}|R^{(j)}_{kl}|^2)^{\frac{4+\varkappa}2}
+\mu_{4+\varkappa}\sum_{l\in\mathbb T_j}(\sum_{k\in\mathbb T_{k,j}}|R^{(j)}_{kl}|^2)^{\frac{4+\varkappa}2}\notag\\&+
\mu_{4+\varkappa}^2\sum_{l\in\mathbb T_j}\sum_{k\in\mathbb T_{k,j}}|R^{(j)}_{kl}|^{{4+\varkappa}}\Bigg)\le 
 Cn^{-(4+\varkappa)}(\sum_{k\in\mathbb T_{k,j}}|R^{(j)}_{kl}|^2)^{\frac{4+\varkappa}2}\notag
\end{align}
Now \cite[Lemma 7.6  inequality (7.11)]{GT:2015} completes the proof. Thus Lemma \ref{eps2} is proved.
\end{proof}
\begin{lem}\label{lambdaopt} Assuming the conditions of Theorem \ref{trun} there exist a positive constant depending on 
$\varkappa$, $D$ such that for any $z\in\mathbb G$
\begin{equation}
\E|\Lambda_n|^2\le \frac C{n^2v^2}.
\end{equation}
\end{lem}
\begin{proof}We write
\begin{align}\notag
 \E|\Lambda_n|^2=\E\Lambda_n\overline \Lambda_n=\E\frac{T_n}{z+m_n(z)+s(z)}\overline\Lambda_n=\sum_{\nu=1}^4\E\frac{T_{n\nu}}{z+m_n(z)+s(z)}\overline\Lambda_n,
\end{align}
where
\begin{align}\notag
 T_{n\nu}:=\frac1n\sum_{j=1}^n\varepsilon_{j\nu}R_{jj}, \text{ for }\nu=1,\ldots,4.
\end{align}
First we observe that by \eqref{shur}
\begin{align}\notag
 |T_{n4}|=\frac1n|m_n'(z)|\le \frac1{nv}\im m_n(z).
\end{align}
 Hence $|z+m_n^{(j)}(z)+s(z)|\ge \im m_n^{(j)}(z)$ and Jensen's inequality yields
\begin{equation}\label{tn4}
 |\E\frac{T_{n4}}{z+m_n(z)+s(z)}\overline\Lambda_n|\le \frac 1{nv}\E^{\frac12}|\Lambda_n|^2.
\end{equation}
Furthermore, we represent $T_{n1}$ as follows
\begin{equation}\notag
 T_{n1}=T_{n11}+T_{n12},
\end{equation}
where
\begin{align}
 T_{n11}&=-\frac1n\sum_{j=1}^n\varepsilon_{j1}\frac1{z+m_n(z)},\notag\\
 T_{n12}&=\frac1n\sum_{j=1}^n\varepsilon_{j1}(R_{jj}+\frac1{z+m_n(z)}).\notag
\end{align}
Using these notations we may write
\begin{align}
 V_1:=\E\frac{T_{n11}}{z+m_n(z)+s(z)}\overline\Lambda_n=-\E\frac{(\frac1n\sum_{j=1}^n\varepsilon_{j1})}{(z+m_n(z))(z+s(z)+m_n(z))}\overline \Lambda_n.\notag
\end{align}
Applying the Cauchy -- Schwartz inequality twice and using the definition of $\varepsilon_{j1}$, we get by \cite[Lemma 7.5 inequality (7.9)]{GT:2015}
\begin{align}
 |V_1|\le \frac1{|z^2-4|^{\frac12}}\E^{\frac14}\Big|\frac1{n\sqrt n}\sum_{j=1}^nX_{jj}\Big|^4\E^{\frac14}|\frac1{z+m_n(z)|^4}\E^{\frac12}|\Lambda_n|^2.
\end{align}
By Rosenthal's inequality, we have, for $z\in\mathbb G$
\begin{equation}\label{v1}
 |V_1|\le \frac C{n|z^2-4|^{\frac12}}\E^{\frac12}|\Lambda_n|^2\le \frac1{nv}\E^{\frac12}|\Lambda_n|^2.
\end{equation}
Using \eqref{rjj1} we  rewrite $T_{n12}$, obtaining
\begin{align}\notag
 V_2:=\E\frac{T_{n12}}{z+m_n(z)+s(z)}\overline\Lambda_n=\frac1{n\sqrt n}\sum_{j=1}^n\E\frac{X_{jj}\varepsilon_jR_{jj}}{(z+m_n(z))(z+m_n(z)+s(z))}\overline\Lambda_n.
\end{align}
By the Cauchy -- Schwartz inequality, 
using the definition of $\varepsilon_j$ (see representation \eqref{rjj}), we obtain
\begin{align}\label{v3}
 |V_2|\le \frac1{\sqrt n}\sum_{\nu=1}^4\E^{\frac12}\frac{|\frac1n\sum_{j=1}^n\varepsilon_{j\nu}X_{jj}R_{jj}|^2}{|z+m_n(z)+s(z)|^2|z+m_n(z)|^2}
 \E^{\frac12}|\Lambda_n|^2=:\sum_{\nu=1}^4V_{2\nu}.
\end{align}
For $\nu=1$, we have
\begin{align}
 V_{21}\le\frac1{ n}\E^{\frac12}\frac{\Big|\frac1n\sum_{j=1}^nX_{jj}^2R_{jj}\Big|^2}{|z+m_n(z)+s(z)|^2|z+m_n(z)|^2}\E^{\frac12}|\Lambda_n|^2.\notag
\end{align}
Applying  H\"older's inequality   and \cite[Lemma 7.5 inequality (7.9)]{GT:2015}, we arrive at
\begin{align}\label{v4}
 V_{21}\le \frac1{ n\sqrt{|z^2-4|}}&\E^{\frac2{4+\varkappa}}\Big(\frac1n\sum_{j=1}^n|X_{jj}|^{\frac{4+\varkappa}2}\Big)^2\notag\\&\times\E^{\frac{\varkappa}{4(4+\varkappa)}}\Big(\frac1n\sum_{j=1}^n|R_{jj}|^{\frac{2(4+\varkappa)}{\varkappa}}\Big)^4
 \E^{\frac{\varkappa}{4(4+\varkappa)}}\frac1{|z+m_n(z)|^{\frac{8(4+\varkappa)}{\varkappa}}}\E^{\frac12}|\Lambda_n|^2.
\end{align}
Observe that
\begin{align}\label{x4}
 \E\Big(\frac1n\sum_{j=1}^n|X_{jj}|^{\frac{4+\varkappa}2}\Big)^2&=\Big(\frac1n\sum_{j=1}^n\E |X_{jj}|^{\frac{4+\varkappa}2}\Big)^2
 +\E\Big(\frac1n\sum_{j=1}^n(|X_{jj}|^{\frac{4+\varkappa}2}-\E|X_{jj}|^{\frac{4+\varkappa}2} )\Big)^2\notag\\&\le 
 2\mu_{4+\varkappa}+\frac2{n^2}\sum_{j=1}^n\E|X_{jj}|^{4+\varkappa}\le  C. 
\end{align}
The last inequality, inequality \eqref{v4}, Theorem \ref{rjj} and \cite[Lemma 7.6]{GT:2015} together imply
\begin{equation}\label{v6}
 V_{21}\le \frac C{n\sqrt{|z^2-4|}}\E^{\frac12}|\Lambda_n|^2\le \frac C{nv}\E^{\frac12}|\Lambda_n|^2.
\end{equation}
Furthermore, for $\nu=4$, by \cite[Lemma 7.16]{GT:2015}, we have
\begin{align}
 V_{24}\le \frac1{nv\sqrt n}\E^{\frac12}\frac{\frac1n\sum_{j=1}^n|X_{jj}|^2|R_{jj}|^2}{|z+m_n(z)+s(z)|^2|z+m_n(z)|^2}
 \E^{\frac12}|\Lambda_n|^2.\notag
\end{align}
Applying H\"older's inequality  and \cite[Lemma 7.5 ]{GT:2015}, we get
\begin{align}
 V_{24}\le\frac1{nv\sqrt n\sqrt{|z^2-4|}}
 \E^{\frac2{4+\varkappa}}\Big(\frac1n\sum_{j=1}^n|X_{jj}|^{\frac{4+\varkappa}2}\Big)&
 \E^{\frac{\varkappa}{4(4+\varkappa)}}\Big(\frac1n\sum_{j=1}^n|R_{jj}|^{\frac{4(4+\varkappa)}{\varkappa}}\Big)\notag\\&
 \E^{\frac{\varkappa}{4(4+\varkappa)}}\frac1{|z+m_n(z)|^{\frac{4(4+\varkappa)}{\varkappa}}}\E^{\frac12}|\Lambda_n|^2.\notag
\end{align}
Applying Theorem \ref{rjj},  we obtain
\begin{equation}\label{v7}
 V_{24}\le  \frac C{nv}\E^{\frac12}|\Lambda_n|^2.
\end{equation}
By H\"older's inequality, we have for $\nu=2,3$,
\begin{align}
 V_{2\nu}&\le \frac1{\sqrt n}\E^{\frac2{4+\varkappa}}\left(\frac1n\sum_{j=1}^n\frac{|\varepsilon_{j\nu}|^{\frac{4+\varkappa}2}|X_{jj}|^{\frac{4+\varkappa}2}}{|z+m_n(z)+s(z)|^{\frac{4+\varkappa}2}}\right)\E^{\frac{\varkappa}{4(4+\varkappa}}\left(\frac1n\sum_{j=1}^n|R_{jj}|^{\frac{8(4+\varkappa)}{\varkappa}}\right)
\notag\\&\times
\E^{\frac{\varkappa}{4(4+\varkappa)}} \frac1{|z+m_n(z)|^{\frac{8(4+\varkappa)}{\varkappa}}}
 \E^{\frac12}|\Lambda_n|^2.\notag
\end{align}
Note that for $\nu=2,3$, r.v. $X_{jj}$ doesn't depend on $\varepsilon_{j\nu}$ and on the $\sigma$-algebra  $\mathfrak M^{(j)}$.
We may write, using \cite[inequality (7.42)]{GT:2015},
\begin{align}
 \E\frac{|\varepsilon_{j\nu}|^{\frac{4+\varkappa}2}|X_{jj}|^{\frac{4+\varkappa}2}}{|z+m_n(z)+s(z)|^{\frac{4+\varkappa}2}}
 &\le C\sqrt{\mu_{4+\varkappa}}\,\E\frac{|\varepsilon_{j\nu}|^{\frac{4+\varkappa}2}}{|z+m_n^{(j)}(z)+s(z)|^{\frac{4+\varkappa}2}}.\notag
\end{align}
Applying now Lemmas \ref{eps2}, \ref{eps3} and \cite[Lemma 7.5]{GT:2015}, we arrive at
\begin{equation}\label{v8}
 V_{2\nu}\le  \frac C{nv}\E^{\frac12}|\Lambda_n|^2,\text{ for }\nu=2,3.
\end{equation}
Inequalities \eqref{v6}, \eqref{v7}, \eqref{v8} together imply
\begin{equation}\label{v9}
 V_2\le \frac C{nv}\E|\Lambda_n|^2.
\end{equation}

Consider now the quantity
\begin{equation}
 Y_{\nu}:=\E\frac{T_{n\nu}}{z+m_n(z)+s(z)}\overline\Lambda_n,\notag
\end{equation}
for $\nu=2,3$.
We represent it as follows
\begin{equation}
 Y_{\nu}=Y_{\nu1}+Y_{\nu2},\notag
\end{equation}
where
\begin{align}
 Y_{\nu1}&=-\frac1n\sum_{j=1}^n\E\frac{\varepsilon_{j\nu}\overline\Lambda_n}{(z+m_n^{(j)}(z))(z+m_n(z)+s(z))},\notag\\
 Y_{\nu2}&=\frac1n\sum_{j=1}^n\E\frac{\varepsilon_{j\nu}(R_{jj}+\frac1{z+m_n^{(j)}(z)})\overline\Lambda_n}{z+m_n(z)+s(z)}.\notag
\end{align}
By the representation \eqref{rjj1},  we have
\begin{align}
 Y_{\nu2}=\sum_{\mu=1}^3\frac1n\sum_{j=1}^n\E\frac{\varepsilon_{j\nu}\varepsilon_{j\mu}\overline\Lambda_nR_{jj}}{(z+m_n(z)+s(z))(z+m_n^{(j)}(z))}.\notag
\end{align}
Using  \cite[inequality (7.42)]{GT:2015}, we may write, for $z\in\mathbb G$
\begin{align}
 |Y_{\nu2}|\le\sum_{\mu=1}^3\frac Cn\sum_{j=1}^n\E\frac{|\varepsilon_{j\nu}||\varepsilon_{j\mu}||\overline\Lambda_n||R_{jj}|}{|z+m_n^{(j)}(z)+s(z)||z+m_n^{(j)}(z)|}. \notag
\end{align}
Applying  the inequality $ab\le \frac12(a^2+b^2)$, we get
\begin{align}
 |Y_{\nu2}|&\le\sum_{\mu=1}^3\frac Cn\sum_{j=1}^n\E\frac{|\varepsilon_{j\mu}|^2|R_{jj}|}{|z+m_n^{(j)}(z)+s(z)||z+m_n^{(j)}(z)|}
 |\Lambda_n|\notag\\&\le\sum_{\mu=1}^3\frac Cn\sum_{j=1}^n\E\frac{|\varepsilon_{j\mu}|^2}{|z+m_n^{(j)}(z)+s(z)||z+m_n^{(j)}(z)|}
|\Lambda_n^{(j)}||R_{jj}|\notag\\&+\sum_{\mu=1}^3\frac Cn\sum_{j=1}^n\E\frac{|\varepsilon_{j\mu}|^2}{|z+m_n^{(j)}(z)+s(z)||z+m_n^{(j)}(z)|}
|\varepsilon_{j4}||R_{jj}|
 =\widehat Y_{1}+\widehat Y_{2}.
\end{align}
Applying H\"older's inequality, we get
\begin{align}
\E\{|\varepsilon_{j\mu}|^2|R_{jj}|\Big|\mathfrak M^{(j)}\}\le \E^{\frac4{4+\varkappa}}\{|\varepsilon_{j\mu}|^{\frac{4+\varkappa}2}\Big|\mathfrak M^{(j)}\}
\E^{\frac{\varkappa}{4+\varkappa}}\{|R_{jj}|^{\frac{4+\varkappa}{\varkappa}}\Big|\mathfrak M^{(j)}\}.
\end{align}
By Lemmas \ref{eps2} and \ref{eps3}, for $\mu=1,2,3$, we have
\begin{align}\label{neq1}
\E\{|\varepsilon_{j\mu}|^2|R_{jj}|\Big|\mathfrak M^{(j)}\}\le 
C\Big(\frac1n+\frac1{n^2}\sum_{l\in\mathbb T_j}|R^{(j)}_{ll}|^2
+\frac1{{nv}}\im m_n^{(j)}(z)\Big)
\E^{\frac{\varkappa}{4+\varkappa}}\{|R_{jj}|^{\frac{\varkappa}{4+\varkappa}}
|\mathfrak M^{(j)}\}.
\end{align}
Conditioning and using inequality \eqref{neq1} and applying Theorem \ref{rjj}, we arrive at
\begin{align}
\E&\frac{|\varepsilon_{j\mu}|^2}{|z+m_n^{(j)}(z)+s(z)||z+m_n^{(j)}(z)|}
|\Lambda_n^{(j)}||R_{jj}|
\le (\frac C{n|z^2-4|^{\frac12}}
+\frac C{nv})\E^{\frac12}|\Lambda_n^{(j)}|^2\notag\\&\le 
(\frac C{n|z^2-4|^{\frac12}}
+\frac C{nv})\E^{\frac12}|\Lambda_n|^2+(\frac C{n^2v|z^2-4|^{\frac12}}
+\frac C{(nv)^2})\notag\\&\le
\frac C{nv}\E^{\frac12}|\Lambda_n|^2+\frac C{(nv)^2}.
\end{align}
The last inequality implies
\begin{equation}\label{fnu2}
 |Y_{\nu2}|\le \frac C{nv}\E^{\frac12}|\Lambda_n|^2+\frac C{(nv)^2}.
\end{equation}

In order to estimate $Y_{\nu1}$ we introduce now the quantity

\begin{equation}
 \Lambda_n^{(j1)}=\frac1n\Tr\mathbf R^{(j)}-s(z)+\frac{s(z)}n+\frac1{n^2}\Tr{\mathbf R^{(j)}}^2s(z).\notag 
\end{equation}
 Recall that
 \begin{align}
  \eta_{j1}&=\frac1n\sum_{l\in\mathbb T_j}[(\mathbf R^{(j)})^2]_{ll},\quad \eta_{j2}=\frac1n\sum_{k\ne l\in\mathbb T_j}X_{jk}X_{jl}[(\mathbf R^{(j)})^2]_{l,k},\notag\\
  \eta_{j3}&=\frac1n\sum_{l\in\mathbb T_j}(X_{jl}^2-1)[(\mathbf R^{(j)})^2]_{ll}.
 \end{align}
Note that
\begin{equation}\label{7.57}
 |\eta_{j1}|\le \frac1n|\Tr(\mathbf R^{(j)})^2].
\end{equation}
We use that (see \cite[equality (7.41)]{GT:2015})
\begin {equation}\label{7.58}
 \varepsilon_{j4}=\frac1n(1+\eta_{j1}+\eta_{j2}+\eta_{j3})R_{jj}.
\end {equation}
Note that
\begin{align}
 \delta_{nj}&=\Lambda_n-{\widetilde\Lambda}_n^{(j)}=-\varepsilon_{j4}-\frac{s(z)}n-\frac1n\eta_{j0}s(z)\notag\\&=
 \frac1n(R_{jj}-s(z))(1+\eta_{j1})+\frac1n(\eta_{j2}+\eta_{j3})R_{jj}.\notag
\end{align}
This yields 
\begin{equation}\label{7.59}
 |\delta_{nj}|\le \frac1n(1+|\eta_{j1}|)|R_{jj}-s(z)|+\frac1n|\eta_{j2}+\eta_{j3}||R_{jj}|
\end{equation}

We represent $Y_{\nu1}$ in the form
\begin{equation}\notag
 Y_{\nu1}=Z_{\nu1}+Z_{\nu2}+Z_{\nu3}+Z_{\nu4},
\end{equation}
where
\begin{align}
 Z_{\nu1}&=-\frac1n\sum_{j=1}^n\E\frac{\varepsilon_{j\nu}{\overline\Lambda}_n^{(j1)}}{(z+m_n^{(j)}(z))(z+m_n^{(j)}(z)+s(z))},\notag\\
 Z_{\nu2}&=\frac1n\sum_{j=1}^n\E\frac{\varepsilon_{j\nu}\overline\delta_{nj}}{(z+m_n^{(j)}(z))(z+m_n(z)+s(z))},\notag\\
 Z_{\nu3}&=\frac1n\sum_{j=1}^n\E\frac{\varepsilon_{j\nu}{\overline\Lambda}_n\varepsilon_{j4}}{(z+m_n^{(j)}(z))(z+m_n^{(j)}(z)+s(z))(z+m_n(z)+s(z))},\notag\\
 Z_{\nu4}&=-\frac1n\sum_{j=1}^n\E\frac{\varepsilon_{j\nu}{\overline\delta}_{nj}\varepsilon_{j4}}{(z+m_n^{(j)}(z))(z+m_n^{(j)}(z)+s(z))(z+m_n(z)+s(z))}.\notag
\end{align}
First,  note that by conditional independence
\begin{equation}\label{7.60}
 Z_{\nu1}=0.
\end{equation}
Using the triangle inequality and \cite[inequality (7.42)]{GT:2015}, we write
\begin{equation}
|Z_{\nu3}|\le \widehat Z_{\nu3}+\widetilde Z_{\nu3},
\end{equation}
where
\begin{align}
\widehat Z_{\nu3}&=\frac 1n\sum_{j=12}^n\E\frac{|\varepsilon_{j\nu}||\Lambda_n^{(j)}||\varepsilon_{j4}|}{|z+m_n^{(j)}||z+m_n^{(j)}(z)+s(z)|^2},\notag\\
\widetilde Z_{\nu3}&=\frac 1n\sum_{j=12}^n\E\frac{|\varepsilon_{j\nu}||\varepsilon_{j4}|^2}{|z+m_n^{(j)}||z+m_n^{(j)}(z)+s(z)|^2}
\end{align}
Using that $|\varepsilon_{j4}|\le 1/nv$ and applying the Cauchy -- Schwartz inequality, we get
\begin{align}
\widetilde Z_{\nu3}\le\frac1{nv}\frac1n\sum_{j=1}^n\E^{\frac12}\left(\frac{|\varepsilon_{j\nu}|^2}{|z+m_n^{(j)}|^2|z+m_n^{(j)}(z)+s(z)|^2}\right)\E^{\frac12}\left(\frac{|\varepsilon_{j4}|^2}{|z+m_n^{(j)}(z)+s(z)|^2}\right).
\end{align}
Conditioning and applying Theorem \ref{rjj} and \cite[Lemma 7.5 inequality (7.0)]{GT:2015}, we obtain, for $z\in\mathbb G$,
\begin{align}
\E^{\frac12}&\left(\frac{|\varepsilon_{j\nu}|^2}{|z+m_n^{(j)}|^2|z+m_n^{(j)}(z)+s(z)|^2}\right)\notag\\&\le \frac C{\sqrt{nv}}\E^{\frac12}\left(\frac{1}{|z+m_n^{(j)}|^2|z+m_n^{(j)}(z)+s(z)|}\right)\le \frac C{\sqrt nv|z^2-4|^{\frac14}}\le C.
\end{align}
According to Lemma \ref{modify}, we have
\begin{align}
\E^{\frac12}\left(\frac{|\varepsilon_{j4}|^2}{|z+m_n^{(j)}(z)+s(z)|^2}\right)\le \frac C{nv}.
\end{align}
The last inequalities together imply, for $z\in\mathbb G$
\begin{equation}
\widetilde Z_{\nu3}\le\frac C{n^2v^2}
\end{equation}
Furthermore, conditioning and applying the Cauchy -- Schwartz inequality, we obtain
\begin{align}
\E\{|\varepsilon_{j4}||\varepsilon_{j\nu}|\Big|\mathfrak M^{(j)}\}&\le \E^{\frac12}\{|\varepsilon_{j\nu}|^2\Big|\mathfrak M^{(j)}\}\E^{\frac12}\{|\varepsilon_{j4}|^2\Big|\mathfrak M^{(j)}\}\notag\\&
\le \frac1{\sqrt nv}(v+\im m_n^{(j)}(z))^{\frac12}\E^{\frac12}\{|\varepsilon_{j4}|^2\Big|\mathfrak M^{(j)}\}.
\end{align} 
Applying  the Cauchy -- Schwartz inequality again and using \cite[Lemma 7.5 inequality (7.9)]{GT:2015}, we get
\begin{align}
 \widehat Z_{\nu3}&\le \frac C{\sqrt{nv}|z^2-4|^{\frac14}}\frac1n\sum_{j=1}^n\E^{\frac14}\frac{|\varepsilon_{j4}|^2}{|z+m_n^{(j)}(z)|^2|z+m_n^{(j)}(z)+s(z)|^2}\E^{\frac12}|\Lambda_n^{(j)}|^2.\notag
\end{align}
The last inequality and Lemma \ref{modify} with $\theta=2$ imply that, for $z\in\mathbb G$,
\begin{equation}\label{znu3}
\widehat Z_{\nu3}\le\frac C{nv}\E^{\frac12}|\Lambda_n|^2+\frac C{n^2v^2}. 
\end{equation}

Furthermore, note that 
\begin{equation}\notag
 |1+\eta_{j1}|\le v^{-1}\im\{z+m_n^{(j)}(z)\}\le\im\{z+m_n^{(j)}(z)+s(z)\}.
\end{equation}
This inequality together with \eqref{7.59} implies that 
\begin{equation}\label{7.62}
 |Z_{\nu4}|\le \widetilde Z_{\nu4}+\widehat Z_{\nu4},
\end{equation}
where
\begin{align}
\widetilde Z_{\nu4}&=\frac1{n^2v}\sum_{j=1}^n\E\frac{|\varepsilon_{j\nu}\varepsilon_{j4}||R_{jj}-s(z)|}{|z+m_n^{(j)}(z)||z+m_n(z)+s(z)|},\notag\\
 \widehat Z_{\nu4}&=\frac1{n^2v}\sum_{j=1}^n\E\frac{|\varepsilon_{j\nu}\varepsilon_{j4}||\eta_{j2}+\eta_{j3}||R_{jj}-s(z)|}{|z+m_n^{(j)}(z)||z+m_n(z)+s(z)|}
\end{align}
By representation $(3.2)$, we have 
\begin{equation}\label{7.63}
 |R_{jj}-s(z)|\le |\Lambda_n||R_{jj}|+|\varepsilon_j||R_{jj}|.\notag
\end{equation}
This implies that
\begin{equation}\label{7.46}
 \widetilde Z_{\nu4}\le \widetilde Z_{\nu41}+\widetilde Z_{\nu42},
\end{equation}
where
\begin{align}
 \widetilde Z_{\nu41}&=\frac1{n^2v}\sum_{j=1}^n\E\frac{|\varepsilon_{j\nu}\varepsilon_{j4}||\Lambda_n||R_{jj}|}{|z+m_n^{(j)}(z)||z+m_n(z)+s(z)|},\notag\\
\widetilde  Z_{\nu42}&=\frac1{n^2v}\sum_{j=1}^n\E\frac{|\varepsilon_{j\nu}\varepsilon_{j4}||\varepsilon_j||R_{jj}|}{|z+m_n^{(j)}(z)||z+m_n(z)+s(z)|}.\notag
\end{align}
Similar to the bound of $\widehat Z_{\nu3}$ we obtain
\begin{equation}\label{7.64}
\widetilde Z_{\nu41}\le \frac C{nv}\E^{\frac12}|\Lambda_n|^2+\frac C{(nv)^2}.
\end{equation}
Note that
\begin{equation}
|\varepsilon_{j\nu}\varepsilon_{j4}||\varepsilon_j|\le2 |\varepsilon_{j1}|^2|\varepsilon_{j4}|+2|\varepsilon_{j2}|^2|\varepsilon_{j4}|+2|\varepsilon_{j3}|^2|\varepsilon_{j4}|+2|\varepsilon_{j4}|^3.
\end{equation}
We may write now
\begin{equation}
\widetilde  Z_{\nu42}\le \breve Z_1+\cdots+\breve Z_4,
\end{equation}
where, for $\mu=1,2,3,4$
\begin{align}
\breve Z_{\mu}=\frac1{n^2v}\sum_{j=1}^n\E\frac{|\varepsilon_{j\mu}|^2|\varepsilon_{j4}||R_{jj}|}{|z+m_n^{(j)}(z)||z+m_n(z)+s(z)|}.
\end{align}

Using that $|\varepsilon_{j4}|\le 1/(nv)$ and \cite[inequality (7.42)]{GT:2015}, we obtain
\begin{align}
\breve Z_{\mu}&\le \frac C{n^2v^2}\frac1n\sum_{j=1}^n\E\frac{|\varepsilon_{j\mu}|^2|R_{jj}|}{|z+m_n^{(j)}(z)||z+m_n^{(j)}(z)+s(z)|}\notag\\&\le 
\frac C{n^2v^2}\frac1n\sum_{j=1}^n\E^{\frac4{4+\varkappa}}
\frac{|\varepsilon_{j\mu}|^{\frac{4+\varkappa}4}}{|z+m_n^{(j)}(z)+s(z)|^{\frac{4+\varkappa}4}}
\E^{\frac{\varkappa}{2(4+\varkappa})}|R_{jj}|^{\frac{2(4+\varkappa)}{\varkappa}}\notag\\&\quad\qquad\qquad\quad\qquad\qquad\times
\E^{\frac{\varkappa}{2(4+\varkappa)}}
\frac1{|z+m_n^{(j)}(z)|^{\frac{2(4+\varkappa)}{\varkappa}}}.
\end{align}
Applying Lemmas \ref{eps2}, \ref{eps3}, \ref{modify}, we obtain for $z\in\mathbb G$,
\begin{equation}\label{breve}
\breve Z_{\mu}\le \frac C{(nv)^2}.
\end{equation}
This impliues that
\begin{equation}\label{7.66}
\widetilde  Z_{\nu42}\le \frac C{(nv)^2}.
\end{equation}

Inequalities \eqref{7.64} and \eqref{7.66} together imply
\begin{equation}\label{7.67}
 |Z_{\nu4}|\le \frac C{nv}\E^{\frac12}|\Lambda_n|^2+\frac C{n^2v^2}. 
\end{equation}

To bound $Z_{\nu2}$ we first apply \cite[ inequality (7,42)]{GT:2015} and obtain
\begin{align}
  |Z_{\nu2}|&\le\frac Cn\sum_{j=1}^n\E\frac{|\varepsilon_{j\nu}||\delta_{nj}|}
  {|z+m_n^{(j)}(z)||z+m_n^{(j)}(z)+s(z)|}.\notag
\end{align}
Furthermore, similarly to the bound of $Z_{\nu4}$ -- inequality \eqref{7.46} -- we may write 
$$
|Z_{\nu2}|\le \widetilde Z_{\nu2}+\widehat Z_{\nu2},
$$
where 
\begin{align}
 \widetilde Z_{\nu2}&=\frac C{n^2}\sum_{j=1}^n\E\frac{|\varepsilon_{j\nu}||R_{jj}-s(z)|}{|z+m_n^{(j)}(z)|},\notag\\
 \widehat Z_{\nu2}&=\frac C{n^2}\sum_{j=1}^n\E\frac{|\varepsilon_{j\nu}||\eta_{j2}+\eta_{j3}||R_{jj}|}{|z+m_n^{(j)}(z)||z+m_n^{(j)}(z)+s(z)|}.\notag
\end{align}
Furthermore,
\begin{align}\label{7.68}
\widetilde Z_{\nu2}&\le \frac C{n^2}\sum_{j=1}^n\E\frac{|\varepsilon_{j\nu}||\varepsilon_j||R_{jj}|}{|z+m_n^{(j)}(z)|}+\frac C{n^2}\sum_{j=1}^n\E\frac{|\varepsilon_{j\nu}||\Lambda_n||R_{jj}|}{|z+m_n^{(j)}(z)|}.
\end{align}
Lemmas $7.15$, $7.16$, $7.22$, inequality $(7.39)$ and Corollary $5.14$ together imply
\begin{equation}\label{7.100}
 \widetilde Z_{\nu2}\le \frac C{nv}\E^{\frac12}|\Lambda_n|^2+\frac C{n^2v^2}.
\end{equation}

Conditioning and applying now H\"older's inequality, we get
\begin{align}\notag
 |\widehat Z_{\nu2}|&\le\frac C{n^2v^2}\frac1n\sum_{j=1}^n\E^{\frac{\varkappa}{4+\varkappa}}
 |R_{jj}|^{\frac{4+\varkappa}{\varkappa}}
 \E^{\frac4{4+\varkappa}}\frac1{|z+m_n^{(j)}(z)|^{\frac{4+\varkappa}4}}.\notag
\end{align}
The last inequality together with Theorem \ref{rjj}  implies
\begin{equation}\label{7.70}
 |\widehat Z_{\nu2}|\le \frac C{n^2v^2}.
\end{equation} 
Combining inequalities \eqref{tn4}, \eqref{v1}, \eqref{v9}, \eqref{fnu2}, \eqref{7.60}, \eqref{znu3}, \eqref{breve}, \eqref{7.66}, \eqref{7.67}, \eqref{7.100} and \eqref{7.70},  we get
\begin{equation}\label{7.69}
 \E|\Lambda_n|^2\le \frac C{nv}\E^{\frac12}|\Lambda_n|^2+\frac C{n^2v^2}.
\end{equation}

Applying \cite[Lemma 7.4]{GT:2015} with $t=2$ and  $r=1$ completes the proof of Lemma \ref{lambdaopt}.

\end{proof}



\end{document}